%% file: convex_polyhedral.tex
\documentclass[a4paper]{amsart}
\usepackage{etex}
\usepackage{amscd,amssymb,amsmath,mathrsfs,latexsym,upgreek,url}
\usepackage{mathtools}
\usepackage[all]{xy}
\usepackage{color,graphicx}
\usepackage{caption,subcaption}
\usepackage{hyperref}
\hypersetup{breaklinks=true}
\usepackage{todonotes}
\usepackage{enumitem}

\def\ov#1{{\overline{#1}}}

\def\wt#1{{\widetilde{#1}}}

\newcommand{\relint}{\operatorname{relint}}

\newcommand{\vol}{\operatorname{vol}}

\renewcommand{\and}{{\quad \text{ and } \quad}}

\newcommand{\hypo}{\operatorname{hypo}}

\newcommand{\CPA}{\operatorname{CPA}}

\newcommand{\PA}{\operatorname{PA}}
\newcommand{\Sk}{\operatorname{Sk}}

\newcommand{\SCPA}{\operatorname{SCPA}}

\newcommand{\WCPA}{\operatorname{WCPA}}

\newcommand{\Hom}{\operatorname{Hom}}
\newcommand{\id}{\operatorname{id}}

\newcommand{\cl}{{\operatorname{cl}}}

\newcommand{\hooklongrightarrow}{\lhook\joinrel\longrightarrow}

\newcommand{\Lip}{\operatorname{Lip}}

\def \B{\mathbb{B}}

\def \R{\mathbb{R}}

\def \Z{\mathbb{Z}}

\def\cV {{\mathcal V}}

\numberwithin{equation}{section}
\theoremstyle{definition}
\newtheorem{defn}[equation]{Definition}

\newtheorem{rem}[equation]{Remark}
\newtheorem{exmpl}[equation]{Example}

\theoremstyle{plain}
\newtheorem{lem}[equation]{Lemma}

\newtheorem{prop}[equation]{Proposition}
\newtheorem{thm}[equation]{Theorem}
\newtheorem{cor}[equation]{Corollary}

\newtheorem{prop-def}[equation]{Proposition-Definition}

\begin{document}

\title{Convex analysis on polyhedral spaces}

\author[Botero]{Ana Mar\'ia Botero}
\address{Fachbereich Mathematik, Universit\"at
  Regensburg. Universit\"atsstr. 31, 93053 Regensburg, Deutschland}
\email{Ana.Botero@mathematik.uni-regensburg.de}
\urladdr{\url{https://homepages.uni-regensburg.de/~boa15169/}}

\author[Burgos Gil]{Jos\'e Ignacio Burgos Gil}
\address{Instituto de Ciencias Matem\'aticas (CSIC-UAM-UCM-UCM3).
  Calle Nicol\'as Ca\-bre\-ra~15, Campus UAB, Cantoblanco, 28049 Madrid,
  Spain} 
\email{burgos@icmat.es}
\urladdr{\url{http://www.icmat.es/miembros/burgos}}

\author[Sombra]{Mart{\'\i}n~Sombra} 
\address{Instituci\'o Catalana de Recerca
  i Estudis Avan\c{c}ats (ICREA). Passeig Llu{\'\i}s Companys~23,
  08010 Barcelona, Spain  \vspace*{-2.5mm}} 
\address{Departament de Matem\`atiques i
  Inform\`atica, Universitat de Barcelona. Gran Via 585, 08007
  Bar\-ce\-lo\-na, Spain} 
\email{sombra@ub.edu}
\urladdr{\url{http://www.maia.ub.edu/~sombra}}

\date{\today} 
\thanks{Botero was partially supported by the SFB Higher Invariants at
  the University of Regensburg. Burgos was partially supported by the
  MINECO research projects MTM2016-79400-P and PID2019-108936GB-C21
  and by the Severo Ochoa 
  programs for centers of excellence SEV-2015-0554 and
  CEX2019-000904-S (ICMAT Severo
  Ochoa).  Sombra was partially supported by the MINECO research
  projects MTM2015-65361-P and PID2019-104047GB-I00 and by the Mar{\'\i}a
  de Maeztu program for units of excellence MDM-2014-0445 (BGSMath
  Mar\'ia de Maeztu).}

\begin{abstract}
  We introduce notions of concavity for functions on balanced polyhedral
  spaces, and we show that concave functions on such spaces satisfy
  several strong continuity properties.
\end{abstract} 

\maketitle

\vspace{-8mm}

\tableofcontents

\section[Introduction]{Introduction}
\label{sec:introduction}

Convex analysis studies properties of convex functions and convex
sets. The notion of convexity is a simple and natural one which has
been increasingly important in both pure and applied mathematics.  One
of the main advantages of working with convex functions is that they
have nice topological properties. For instance, a convex function on a
open subset $U\subset\R^{n}$ is continuous on $U$, and Lipschitz
continuous on every compact subset of it. The aim of this article is
to transfer the notion of convexity of functions to balanced
polyhedral spaces, and to extend to this setting the strong continuity
properties of convex functions on open subsets of~$\R^{n}$.

There are algebraic objects to which one can naturally attach a
balanced polyhedral space. The results of this article can be applied
directly to such spaces, and we hope that this serves to a better
understanding of the algebraic objects involved. Important examples
arise in tropical geometry. Here, the procedure of tropicalization
attaches a tropical cycle to an algebraic cycle, and the tropical
cycle has a structure of a balanced polyhedral space \cite{AR}.

Related examples arise in the theory of toroidal embeddings. As a
particular case, to a pair $(X, D)$ consisting of an algebraic variety
$X$ and a simple normal crossings divisor $D$ on $X$ one can associate
its Clemens complex (also called the dual complex), which is also a
balanced polyhedral space \cite{GR, BB19:TbMA}.

In \cite{AR}, Allerman and Rau developed an intersection theory
between piecewise affine functions and tropical cycles. This theory
was recently extended by Gross to a tropical intersection theory on
weakly embedded conical polyhedral complexes associated to toroidal
embeddings, between \emph{combinatorially principal} piecewise affine
functions and tropical cycles \cite{GR}.

A \emph{balanced polyhedral space} $X$ is an (abstract) polyhedral
space of pure dimension $n$, which we equip with additional structure
that allows to define a balancing condition. More precisely, it can be
identified with a $5$-tuple $(X, \Pi, N, \iota, b)$, where $X$ is a
second countable $n$-dimensional topological space, $\Pi $ is a
polyhedral complex on $X$, $N$ is a Euclidean vector space, $\iota$ is
a continuous map $\iota \colon X \to N$ whose restriction to each
polyhedron $\sigma \in \Pi $ is injective and affine, and $b$ is an
$n$-dimensional Minkowski weight on $\Pi$ satisfying $b(\sigma) > 0$
for every $n$-dimensional cone $\sigma \in \Pi$. The map $\iota$ is
called the quasi-embedding and $b$ the balancing condition
(Definition~\ref{def:balanced}).

The additional structure considered in the present article differs
from the ones considered in \cite{AR} and \cite{GR} in the following
two aspects: on the one side, the authors in \emph{loc. cit.} consider
an integral structure whereas in this article, a Euclidean structure
is considered. This turns out to be more convenient when dealing with
convexity notions. On the other side, a tropical variety as in
\cite{AR} comes with an embedding into a real vector space, while a
weak embedding as in \cite{GR} does not need to be injective when
restricted to a cone. By contrast, a quasi-embedding does not need to
be globally injective but its restriction to every polyhedron is
assumed to be injective.  A consequence of this difference is that we
do not have to care about piecewise affine functions being
combinatorially principal when doing intersection theory. Indeed,
injectivity on each polyhedra implies that any piecewise affine
function on the complex is combinatorially principal.

We will work with \emph{concave} functions instead of
convex ones, because the applications we have in mind come from
toroidal geometry and positive divisors in this theory are described
by concave functions on balanced polyhedral spaces~\cite{BB19:TbMA}.
Indeed, the functions associated to positive divisors in toric geometry
are concave in the sense of convex analysis, although they are usually
called convex in the toric literature \cite{Fulton:itv}.

A first approach to concavity on polyhedral spaces is to declare that
a function $f$ on a balanced polyhedral space $X$ is concave if it is
the pullback by the quasi-embedding $\iota$ of a concave function on
the Euclidean space $N$. But this notion of concavity is not optimal
because it is not local (Example \ref{exm:8}).  The variant that
declares $f$ to be concave if it is locally of the above type, is not
stable under the operation of taking the infimum of a family of such
functions. More crucially, given a toroidal embedding, the piecewise
affine function on the Clemens complex associated to a nef toroidal
divisor on the algebraic variety is not necessarily concave in this
sense.

Hence it makes sense to explore other notions of concavity. In this
article, we give several notions of concavity in terms of convex
combinations.  A \emph{convex combination} of points in $X$ is a
triple
  \begin{displaymath}
   (x,\{x_{i}\}_{i\in I}, \{\nu _{i}\}_{i\in I})  
 \end{displaymath}
 where $I$ is a finite set, $x\in X$ is the central point, the
 $x_{i}$'s are also points in $ X$ and the $\nu _{i}$'s are
 nonnegative real numbers such that
 \begin{displaymath}
\sum_{i\in I} \nu _{i}=1 \and   \iota (x)=\sum_{i\in I}\nu _{i}\, \iota(x_{i}).
 \end{displaymath}
 A convex combination is \emph{polyhedral} if there is a polyhedral
 structure $\Pi $ on $X$ such that the central point $x$ belongs to a
 polyhedron $\tau $ and the remaining points belong to polyhedra
 having $\tau $ as a face.  Finally, a convex combination is
 \emph{balanced} if it is polyhedral and, roughly speaking, the
 location of the $x_{i}$'s is dictated by the balancing condition
 (Definition \ref{def:5}).

Then, a function $f$ on $X$ is  \emph{strongly
  concave} if for all convex 
combinations, the concavity condition
\begin{equation}
  \label{eq:34}
  f(x)\ge \sum_{i\in I}\nu _{i} \, f(x_{i})
\end{equation}
is satisfied. This is equivalent to asking that $f$ is the pullback by
the quasi-embedding $\iota$ of a concave function on the Euclidean
space $N$ (Proposition \ref{prop:17}).  The function $f$ is 
\emph{concave} if the concavity condition \eqref{eq:34} is satisfied
for all polyhedral convex combination, and it is \emph{weakly
  concave} if this concavity condition is only assumed to hold for
balanced convex combinations (Definition \ref{def:1}).

Clearly, strongly concave functions are concave and, in turn, concave
functions are weakly concave. The three notions of concavity are
different (Examples \ref{exm:4}, \ref{exm:7} and \ref{exm:9}) but,
when restricted to $\R^{n}$, they all agree with the usual one.
Moreover, in the piecewise affine case, these notions of concavity can
be reformulated in terms of intersection products with Minkowski
cycles (Proposition \ref{prop:17}).

The main results of this article show that weakly concave functions
(hence also concave and strongly concave functions) on a balanced polyhedral
space satisfy the same nice continuity properties satisfied by concave
functions on $\R^{n}$.

The first result in this direction is Theorem \ref{thm:2} that states
that, if $f$ is a weakly concave function on an open subset
$U\subset X$, then it is continuous. Even more, in Corollary
\ref{cor:lipschitz} we see that such a function is Lipschitz
continuous in every compact subset $K\subset U$. More precisely, the
Lipschitz constant of $f$ in $K$ can be bounded in terms of the
sup-norm of the function on a slightly bigger open subset
(Theorem~\ref{thm:3}).

Once we have proven Lipschitz continuity on compact subsets, we can
prove some strong uniformity and convergence results for families of
weakly concave functions (Theorems \ref{thm:conv1}, \ref{thm:4} and
\ref{thm:conv2}). For instance, Theorem \ref{thm:4} states that, if
$(f_{i})_{i\ge 0}$ is a sequence of weakly concave functions that
converge pointwise on a dense subset to finite limits, then the
sequence of functions converges pointwise to a weakly concave function
and the convergence is uniform on compact subsets.

As we mentioned earlier, our motivation comes from toroidal geometry.
In the subsequent article by the first and the second authors
\cite{BB19:TbMA}, the strong continuity properties of concave
functions in open subsets of polyhedral complexes are used to show
that the degree of a nef toroidal b-divisor is well-defined and to
prove a Hilbert-Samuel type theorem for such b-divisors.  However, we
believe that the results of this article are general enough to be of
interest in other areas of mathematics.

The article is organized as follows. In Section
\ref{sec:polyhedral-complexes} we gather several definitions
concerning polyhedral spaces, including the key notions of
quasi-embedded and of Euclidean polyhedral spaces and complexes. We
define morphisms between these spaces and show that they satisfy nice
topological properties (Proposition \ref{prop:8}).

In Section \ref{sec:mink-weights-inters} we introduce the space of
Minkowski weights on a Euclidean polyhedral complex (Definition
\ref{def:4}) and the space of Minkowski cycles on a Euclidean
polyhedral space (Definition \ref{def:14}). We also introduce
{balanced polyhedral spaces}, that is, Euclidean polyhedral spaces
endowed with a full-dimensional strictly positive Minkowski cycle
(Definition \ref{def:balanced}). We end this section by defining an
intersection product between Minkowski cycles and piecewise affine
functions, and we show that it is well-defined and symmetric
(Proposition \ref{prop:4}). This intersection product is the Euclidean
translation of that in \cite{Rau} and \cite{GR}.

In Section \ref{sec:conc-conv-funct} we discuss different notions of
concavity for piecewise affine functions related to preserving
the positivity of the intersection product. We end this section by
introducing the notions of regular polyhedral structures on a balanced
polyhedral space $X$ and of
strictly concave function. This is inspired by the
correspondence between normal fans of polytopes and projective toric
varieties. We then show the analogue of the toric Chow lemma from
toric geometry, providing existence of regular structures whenever $X$
satisfies a natural finiteness condition (Theorem~\ref{prop:projective}).

In Section \ref{sec:conc-funct-polyh} we discuss the different notions
of concavity of functions on $X$ that are not necessarily piecewise
affine, in terms of convex combinations of points. We show that in the
piecewise affine case, these notions agree with those introduced in
the previous section (Proposition \ref{prop:17}).

Finally, in Section \ref{sec:cont-prop} we prove the strong continuity
properties of weakly concave (and hence also of concave and strongly
concave) functions on a balanced polyhedral space.

\medskip \noindent {\bf Acknowledgments.} We are grateful to
  the anonymous referee for all her/his constructive remarks.

\section{Polyhedral spaces}
\label{sec:polyhedral-complexes}

In this section we gather several definitions concerning polyhedral
and complexes. For convex sets and functions we will use the notations
and definitions in \cite{ROCK}. In particular, the empty set
  is a face of every convex set.

\begin{defn} \label{def:10} Let $X$ be a second countable topological
  space.  A \emph{polyhedral structure} on $X$ is a pair
\[
\Pi = (\{\sigma_{\alpha}\}_{\alpha \in \Lambda},
  \{M_{\alpha} \}_{\alpha \in \Lambda})
\]
consisting of
\begin{enumerate}
\item \label{item:21} a collection
  $\{\sigma_{\alpha}\}_{\alpha \in \Lambda}$ of different closed
  subsets giving a locally finite covering of~$X$,
\item \label{item:22} a collection
  $\{M_{\alpha} \}_{\alpha \in \Lambda}$ where each $M_{\alpha}$ is a
  finite dimensional vector space of continuous $\R$-valued functions on
  $\sigma_{\alpha}$ such that, setting
  $N_{\alpha }= \Hom(M_{\alpha}, \R)$, the evaluation map
  \begin{displaymath}
\phi_{\alpha} \colon \sigma_{\alpha} \longrightarrow  N_{\alpha }
  \end{displaymath}
 sends $\sigma_{\alpha}$ homeomorphically onto a
    full-dimensional convex polyhedron of a hyperplane $H_{\alpha}$ of
    this dual vector space not passing through the origin.
\end{enumerate}
The closed subsets $\sigma_{\alpha}$ are the \emph{polyhedra} of
$\Pi$, and the preimages under $\phi_{\alpha}$ of the faces of
$\phi_{\alpha}(\sigma_{\alpha})$ are the \emph{faces} of
$\sigma _{\alpha }$. We assume that the pair $\Pi$ satisfies the
conditions:
\begin{enumerate}
  \setcounter{enumi}{2}
\item\label{item:evaluationmap} every face of a polyhedron
  $\sigma_{\alpha}$ of $\Pi$ is a polyhedron $\sigma_{\alpha'}$ of
  $\Pi$ for an index $\alpha' \in \Lambda$, and the corresponding
  vector space $M_{\alpha'}$ consists of the restriction to
  $\sigma_{\alpha'}$ of the functions in $M_{\alpha}$,
\item \label{item:20} every two polyhedra of $\Pi$  intersect
  in a common face (possibly the empty one). 
\end{enumerate}
  A polyhedral structure is \emph{finite} if its set of polyhedra is
  finite.
\end{defn}

For simplicity, we identify a polyhedral structure $\Pi $ with its
underlying set of polyhedra.  For a polyhedron $\sigma\in \Pi$ we
write
\begin{displaymath}
M_{\sigma}, \quad N_{\sigma}, \quad H_{\sigma} \and \phi_{\sigma}
\end{displaymath}
for its corresponding vector space, dual space, hyperplane and
evaluation map, respectively.

Identifying each $\sigma\in \Pi$ with its image in $N_{\sigma }$, we
can translate to it the objects and notions of polyhedra in vector
spaces. In particular, we denote by $\relint (\sigma ) $ the
\emph{relative interior} of $\sigma$, that is, the preimage under
$\phi_{\sigma }$ of the interior of the polyhedron
$\phi_{\sigma }(\sigma) \subset N_{\sigma }$.  Likewise the affine
structure of $N_{\sigma }$ gives an affine structure on $\sigma$, and
therefore we can talk about affine maps between polyhedra.

\begin{rem}
    The condition \eqref{item:22} in Definition \ref{def:10} implies
    that the space of affine functions on $\sigma$ coincides with
    $M_{\sigma}$.
  \end{rem}

\begin{defn}
  \label{def:16}
  Let $X$ be a second countable topological space and $\Pi, \Pi'$ 
  two polyhedral structures on $X$. Then $\Pi '$ is a \emph{subdivision} of $\Pi $, denoted by
  $\Pi '\ge \Pi$, if for every $\sigma' \in \Pi'$ there
  exists a $\sigma \in \Pi$ with
\begin{math}
  \sigma'\subset \sigma,
\end{math}
the inclusion being an affine map.  Two polyhedral structures on $X$
are \emph{equivalent} if they admit a common subdivision.
\end{defn}

  \begin{prop}\label{prop:2}
  Let $X$ be a second countable topological space. Then
  \begin{enumerate}
  \item \label{item:25} the relation $\ge$ is a partial order on the
    set of polyhedral structures on $X$,
  \item \label{item:23} the subdivisions of a given polyhedral
    structure on $X$ form a directed set,
  \item \label{item:24} ``being equivalent'' is an equivalence
    relation between polyhedral structures on $X$.
  \end{enumerate}  
\end{prop}

\begin{proof}
  The statement in \eqref{item:25} is clear from the definitions. In
  particular, the antisymmetric property follows from the fact that
  the polyhedra in a given polyhedral structure are different, and so
  two polyhedral structures that subdivide each other are necessarily
  equal.

  To prove the statement in \eqref{item:23}, let $\Pi$ be polyhedral
  structure on $X$, and let $\Pi_{1}$ and $\Pi_{2}$ be two
  subdivisions of $\Pi$.  For $i=1,2$ and each $\tau_{i}\in \Pi_{i}$
  there is $\sigma_{i}\in \Pi$ such that $\tau_{i}\subset
  \sigma_{i}$. We have that $\sigma=\sigma_{1}\cap \sigma_{2}$ is a
  polyhedron of $\Pi$ containing the intersection
  $ \tau_{1}\cap\tau_{2}$, and for each $i$ we have that
  $\tau_{i}\cap \sigma$ is a polyhedron of the vector space
  $N_{\sigma}$. Hence
  \begin{displaymath}
    \tau_{1}\cap\tau_{2} = (\tau_{1}\cap\sigma) \cap(\tau_{2}\cap\sigma)
  \end{displaymath}
  is also a polyhedron of $N_{\sigma}$, and its space of affine
  functions $ M_{\tau_{1}\cap\tau_{2}}$ consists of the functions in
  $M_{\sigma}$ restricted to it. The pair
  \begin{displaymath}
    \left(\{\tau_{1}\cap\tau_{2}\}_{\tau_{1}, \tau_{2}}, \{M_{\tau_{1}\cap\tau_{2}}\}_{\tau_{1}, \tau_{2}}\right)
  \end{displaymath}
  is a polyhedral structure on $X$ that subdivides both $\Pi_{1}$ and
  $\Pi_{2}$, and so the subdivisions of $\Pi$ form a directed set, as
  stated.
  
  The relation ``being equivalent'' is clearly both symmetric and and
  reflexive.  To check the transitivity property, let $\Pi $, $\Pi '$
  and $\Pi ''$ be three polyhedral structures on $X$ such that $\Pi '$
  and $\Pi ''$ are both equivalent to $\Pi $.  Then there are
  polyhedral structures $\Pi_{1}$ and $\Pi_{2}$ that are a common
  subdivision of $\Pi $ and $\Pi '$ and of $\Pi $ and $\Pi ''$,
  respectively.  Since $\Pi_{1}$ and $\Pi_{2}$ are subdivisions of
  $\Pi$, by \eqref{item:23} there is a further polyhedral structure
  $\Pi_{3}$ that subdivides both of them, and by \eqref{item:25} we
  have that $\Pi_{3}$ subdivides both $\Pi'$ and $\Pi''$. Hence the
  later polyhedral structures are equivalent, proving \eqref{item:24}.
\end{proof}

\begin{defn}
  A \emph{polyhedral space} $X$ is a second countable topological
  space equipped with an equivalence class of polyhedral structures. A
  \emph{polyhedral complex on} $X$ is the choice of a representative
  of the class of polyhedral structures of $X$. 
\end{defn}

\begin{rem}
By Proposition \ref{prop:2}, the set of polyhedral complexes on a
polyhedral space is a directed set ordered by subdivision.
\end{rem}

\begin{defn}
  The \emph{dimension} of a polyhedral space $X$ is defined
  as
  \begin{displaymath}
    \dim(X)=   \sup_{\sigma \in \Pi } \dim \left(M_{\sigma }\right)-1
  \end{displaymath}
  for any polyhedral complex $\Pi$ on $X$.  We say that $X$ has
  \emph{pure dimension $n$} when every polyhedron of $\Pi$ that is
  maximal (with respect to the inclusion) has dimension~$n$. These
  notions do not depend on the choice of $\Pi$.
\end{defn}

\begin{defn} \label{def:15} Let $X$ and $X'$ be polyhedral spaces.
  Given polyhedral complexes $\Pi$ on $X$ and $\Pi'$ on $X'$, a
  \emph{morphism of a polyhedral complexes} between $\Pi$ and $\Pi'$
  is a continuous map $f\colon X\to X'$ such that for every
  $\sigma \in \Pi$ there is $\sigma' \in \Pi'$ with
  $f(\sigma) \subset \sigma'$, and the restriction
  $f|_{\sigma} \colon \sigma \to \sigma'$ is an affine map.

  A \emph{morphism of polyhedral spaces} between $X$ and $X'$ is a
  continuous map $f\colon X\to X'$ that becomes a morphism of
  polyhedral complexes after a suitable choice of polyhedral complexes
  $\Pi$ on $X$ and $\Pi'$ on $X'$ as above.
\end{defn}

\begin{prop}\label{prop:8}
  The underlying topological space of a polyhedral space is Hausdorff.
\end{prop}

\begin{proof} Let $X$ be a polyhedral space and $\Pi $ a polyhedral
  complex on $X$.  Let $x, y\in X$ be two different points. Since the
  covering $\{\sigma\}_{\sigma \in \Pi}$ is locally finite, we can
  choose an open subset $W$ of $X$ containing both $x$ and $y$ and
  such that the set of polyhedra
\begin{displaymath}
  \Sigma=\{\sigma\in \Pi \mid \sigma \cap W \ne \emptyset\}
\end{displaymath}
is finite.

Each $\sigma\in\Sigma$ is homeomorphic to a polyhedron and so a
Hausdorff topological subset of $X$. Hence there are open subsets
$U_{\sigma}$ and $V_{\sigma} $ of $X$ that are disjoint on $ \sigma$,
and contain $x$ and $y$ respectively, and so
\begin{displaymath}
  W \cap \bigcap_{\sigma\in\Sigma} U_{\sigma} \and  W \cap \bigcap_{\sigma\in \Sigma} V_{\sigma}  
\end{displaymath}
are disjoint neighborhoods of $x$ and $y$ respectively, proving the
statement.
\end{proof}

A standard way to construct polyhedral spaces is by
gluing polyhedra through their faces, as we explain in the next
example. 

\begin{exmpl} 
  \label{exm:5}
  Let $\Lambda $ be a countable set.  For $\alpha \in \Lambda$, let
    $\Delta_{\alpha}$ be a polyhedron in a finite dimensional vector space $P_{\alpha}$. For each $\alpha,\beta \in \Lambda$, let
  $F_{\alpha,\beta }$ and $F_{\beta,\alpha}$ be faces of
  $ \Delta_{\alpha }$ and $ \Delta_{\beta }$ respectively (possibly
  the empty ones) and
\begin{displaymath}
  j_{\beta,\alpha }\colon F_{\alpha,\beta }\longrightarrow F_{\beta,\alpha }  
\end{displaymath}
 an affine isomorphism. We assume that this data verifies the conditions
\begin{enumerate}
\item \label{item:18} for $\alpha \in \Lambda $, $F_{\alpha ,\alpha }=\Delta_{\alpha
  }$ and $j_{\alpha ,\alpha }=\id_{\Delta _{\alpha }}$,
\item\label{item:19} for $\alpha ,\beta ,\gamma \in \Lambda $, we have that 
  \begin{math}
    j_{\beta,\alpha }(F_{\alpha ,\beta }\cap F_{\alpha,\gamma})=F_{\beta ,\alpha }\cap F_{\beta,\gamma }
  \end{math}
 and
\begin{displaymath}
  j_{\gamma ,\alpha }=j_{\gamma,\beta  }\circ j_{\beta
    ,\alpha} \text{ on } F_{\alpha ,\beta }\cap F_{\alpha,\gamma},
\end{displaymath}
\item \label{item:6} for $\alpha \in \Lambda $, the set $\{\beta \in \Lambda
  \mid F_{\beta ,\alpha}\not = \emptyset \}$ is finite. 
\end{enumerate}
Note that \eqref{item:18} and \eqref{item:19} imply that, for $\alpha
,\beta \in \Lambda $, we have that $j_{\alpha ,\beta
}=j_{\beta ,\alpha }^{-1}$,

For $x \in \Delta_{\alpha}$ and $y\in \Delta_{\beta}$, we set
$x\sim y$ whenever $x\in F_{\alpha,\beta }$, $y\in F_{\beta,\alpha}$
and $j_{\beta ,\alpha }(x)=y$. This defines an equivalence relation on
the disjoint union $\bigsqcup_{\alpha \in \Lambda} \Delta_{\alpha}$,
and we consider the quotient topological space
\begin{equation}\label{eq:25}
  Y = \Big( \bigsqcup_{\alpha \in \Lambda} \Delta_{\alpha} \Big) \Big/ \sim.
\end{equation}
Since $Y$ is equipped with the quotient topology, the map
$q\colon \bigsqcup_{\alpha \in \Lambda} \Delta_{\alpha}\to Y$ is
continuous.  The conditions above imply that the equivalence relation
$\sim$ is closed and that the map $q$ is open. By \cite[Chapter~I, \S
8.3, Proposition~8]{bourbaki-topologie-generale-1-4} we deduce that
$Y$ is Hausdorff. Moreover, the finiteness condition \eqref{item:6}
implies that the quotient map $q$ is proper.

The image in $Y$ of each $\Delta_{\alpha}$ is a closed subset that is
homeomorphic to it. Identifying each of these polyhedra with its image
in $Y$ and defining $M_{\alpha}$ as the space of affine functions on
$\Delta_{\alpha}$, we have that
\begin{equation}
  \label{eq:15}
\Gamma = (\{\Delta_{\alpha}\}_{\alpha \in \Lambda},
  \{M_{\alpha} \}_{\alpha \in \Lambda})
\end{equation}
is a polyhedral structure on $Y$ in the sense of Definition
\ref{def:10}. Indeed, the condition~\eqref{item:6} implies
that the covering $\{\Delta_{\alpha}\}_{\alpha \in \Lambda}$ is
locally finite.
\end{exmpl}

\begin{prop}\label{prop:10}
  Every polyhedral complex on a polyhedral space is isomorphic to one
  constructed gluing polyhedra through faces as in Example
  \ref{exm:5}.
\end{prop}

\begin{proof}
  With notation as in Definition \ref{def:10}, let
  $\Pi = \left(\{\sigma_{\alpha}\}_{\alpha \in \Lambda}, \{M_{\alpha}
    \}_{\alpha \in \Lambda}\right)$ be a polyhedral complex on a
  polyhedral space $X$.  For each $\alpha \in \Lambda$ consider the
  polyhedron in the dual space $N_{\alpha}=\Hom(M_{\alpha},\R)$ given
  by
\begin{displaymath}
\Delta_{\alpha} = \phi_{\alpha}(\sigma_{\alpha})
\end{displaymath}
and, for $\alpha ,\beta \in \Lambda $,  consider the faces of
$\Delta_{\alpha}$ and $\Delta_{\beta}$ respectively defined as
\begin{displaymath}
  F_{\alpha,\beta }=\phi_{\alpha}(\sigma_{\alpha}\cap \sigma _{\beta
  }) \and F_{\beta,\alpha  }=\phi_{\beta }(\sigma_{\alpha}\cap \sigma _{\beta
  }) 
\end{displaymath}
and the affine map
$ j_{\beta ,\alpha } \colon F_{\alpha,\beta} \to F_{\beta,\alpha}$
given by the restriction of $\phi_{\beta }\circ (\phi_{\alpha })^{-1}$
to the face $ F_{\alpha,\beta }$.

This data satisfies the conditions in Example \ref{exm:5} and in
particular, the hypothesis that the covering
$\{\sigma_{\alpha }\}_{\alpha\in \Lambda}$ is locally finite implies
the condition \eqref{item:6}.  Hence we can consider its associated
polyhedral complex $\Gamma$ on the polyhedral space $Y$ as in
\eqref{eq:15} and \eqref{eq:25}.  The isomorphisms
$\phi_{\alpha}^{-1}\colon \Delta_{\alpha}\to\sigma_{\alpha}$,
$\alpha\in\Lambda$, induce a bijective map
\begin{displaymath}
f \colon Y \longrightarrow X
\end{displaymath}
that is continuous, by the universal property of the quotient topology.

Since the set of polyhedra $\sigma_{\alpha}$, $\alpha \in \Lambda$,
forms a locally finite covering of $X$ by closed subsets and the
restriction $f^{-1}|_{\sigma_{\alpha}}=\phi_{\alpha}$ to each of them
is continuous, the inverse map $f^{-1}$ is also continuous. Hence $f$
is a homeomorphism that is affine between each pair of polyhedra
$\Delta_{\alpha}$ and $\sigma_{\alpha}$, and so it is an isomorphism
of polyhedral complexes.
\end{proof}

To do convex analysis on a polyhedral space, we need a notion encoding
how its different polyhedra are placed with respect to each other.
For this, we first need to map it to a fixed ambient space. The
definition below is a variant of \cite[Definition~2.1]{GR}.

\begin{defn}\label{weakembedding}
  A \emph{quasi-embedded polyhedral space} is a triple $(X,N,\iota)$
  where $X$ is a polyhedral space, $N$ a finite dimensional
  $\R$-vector space, and $\iota $ a map $ X\to N$ such that
  $\iota (X)$ is contained in an affine hyperplane $H$ not containing
  zero and there is a polyhedral complex on $X$ for which the
  restriction of $\iota $ to each of its polyhedra is affine and
  injective.  The map $\iota $ is called the \emph{quasi-embedding} of
  $X$ in $N$.

  A \emph{polyhedral complex} $\Pi$ on the quasi-embedded polyhedral
  space $(X,N,\iota)$ is a polyhedral complex on $X$ satisfying the
  above condition, namely that $\iota$ is affine and injective on each
  of its polyhedra. For each $\sigma\in \Pi$, its image with respect to
  the evaluation map $\phi_{\sigma}$ spans a full-dimensional
  polyhedral cone of $N_{\sigma}$, and so the quasi-embedding $\iota$
  induces a linear injective map
  \begin{displaymath}
    \iota _{\sigma }\colon N_{\sigma}\to N
  \end{displaymath}
  that sends $H_{\sigma }$ to $H$.
\end{defn}

We will usually denote a quasi-embedded polyhedral space by its
underlying polyhedral space $X$ and, in this case, we will denote the
corresponding quasi-embedding, vector space and hyperplane by
$\iota _{X}$, $N_{X}$ and $H_{X}$, respectively.

\begin{defn}\label{def:euclidean}
  A \emph{Euclidean polyhedral space} is a quasi-embedded polyhedral
  space $X$ for which the vector space $N_{X}$ is equipped with a
  Euclidean metric.

  A \emph{polyhedral complex} on a Euclidean polyhedral space is a
  polyhedral complex $\Pi$ on the associated quasi-embedded polyhedral
  space. In this situation, for each polyhedron 
  $\sigma \in \Pi $ the Euclidean metric on $N_{X}$ induces a
  Euclidean metric on the vector space $N_{\sigma }$.
\end{defn}

\begin{defn}
  \label{def:17}
  A morphism between two quasi-embedded polyhedral spaces
  $f\colon X\to X'$ is a pair $f=(f_{1},f_{2})$ where
  $f_{1}\colon X \to X'$ is a morphism of polyhedral spaces and
  $f_{2}\colon N_{X}\to N_{X'}$ is an affine map such that the diagram
  \begin{displaymath}
    \xymatrix{X \ar[r]^{f_{1}}\ar[d]_-{\iota_{X}} & X'\ar[d]^-{\iota_{X'} }\\
    N_{X}\ar[r]_{f_{2}} & N_{X'}  
    }
  \end{displaymath}
  commutes. A morphism between Euclidean polyhedral spaces is a
  morphism between their underlying quasi-embedded polyhedral spaces.
\end{defn}

\begin{rem}\label{rem:3}
  The main differences between the notions of quasi-embedded
  polyhedral complex in Definition \ref{weakembedding} and that of
  weakly embedded polyhedral complex in \cite[Definition~2.1]{GR} are
  that for the latter, first the affine maps $\iota_{\sigma }$ are not
  required to be injective and second, the vector spaces $M$ and
  $M_{\sigma }$ are equipped with a lattice and the affine maps
  $\iota_{\sigma }$ are lattice maps.  Here we shift the focus from
  lattices to Euclidean metrics because it is more convenient to do
  convex analysis.
\end{rem}

\section{Minkowski weights and Minkowski cycles}
\label{sec:mink-weights-inters}

In this section we introduce Minkowski weights on polyhedral complexes
and Minkowski cycles on Euclidean polyhedral spaces. This will allow
us to define balanced polyhedral spaces, the spaces on which we
consider the different notions of concavity. We also define and study
the basic operations on Minkowski weights and cycles, including their
restriction to open subsets, pullback to subdivisions, and product
with piecewise affine functions.

Throughout this section, we denote by $X$ a Euclidean polyhedral space
and $U$ an open subset of it.

\begin{defn}
  \label{def:18}
  Let $\Pi $ be a polyhedral complex on $X$. The \emph{restriction} of
  $\Pi $ to $U$, denoted by $\Pi |_{U}$, is the set of polyhedra given
  by
\[
\Pi|_U = \{ \sigma \in \Pi \; \big{|} \; \sigma \cap U \not = \emptyset \}.
\]
For $k\in \Z_{\ge 0}$, we denote by $\Pi|_{U} (k)$ the set of polyhedra of
$\Pi|_{U}$ of dimension $k$.  The \emph{skeleton} of $\Pi|_{U}$ of
dimension $k$, denoted by $\Sk_{k}(\Pi|_{U})$, is the subset of $X$
given by the union of these polyhedra.
\end{defn}

The restriction of a polyhedral complex to an open subset is not a
polyhedral complex, because it does not contain all the faces of its
constituent polyhedra. For instance, if $X=[0,1]$ is the unit
interval, $\Pi =\{[0,1],\{0\},\{1\},\emptyset\}$ is the standard
polyhedral complex on $[0,1]$ and $U=[0,1/2)$, then
$\Pi|_{U} =\{[0,1],\{0\}\}$ and the faces $\emptyset$ and
  $\{1\}$ of $[0,1]$ do not belong to $\Pi|_{U}$.

\begin{defn}
  \label{def:19}
  Let $\Pi $ be a polyhedral complex on $X$ and let
  $\sigma, \tau\in \Pi$ such that $\tau$ is nonempty and a
  \emph{facet} of $\sigma$, that is, a face of $\sigma$ of
  codimension~1. With notation as in Definition \ref{weakembedding},
  the affine subspace $\iota_{X}(H_{\sigma})$ of the Euclidean space
  $N_{X}$ contains $\iota_{X}(H_{\tau})$ as a hyperplane.  The
  \emph{unit vector normal to $\tau$ in the direction of $\sigma$},
  denoted by $v_{\sigma \setminus \tau }$, is defined as the unique
  unit vector in $N_{X}$ that is orthogonal to $\iota_{X}(H_{\tau})$,
  parallel to $\iota_{X}(H_{\sigma })$ and points towards $\sigma $
  from $\tau $.
\end{defn}
  
The next definition of Minkowski weights is the adaptation to our
setting of the classical notion for lattice fans introduced by Fulton
and Sturmfels in \cite{FS}.  For $\sigma, \tau\in \Pi$, we write
either $\tau \prec \sigma $ or $\sigma \succ \tau $ to indicate that
$\tau $ is a face of~$\sigma $.
  
\begin{defn}\label{def:4}
  Let $\Pi$ be a polyhedral complex on $X$ and $k\in \Z_{\ge 0}$. A
  \emph{weight} on $\Pi|_U$ of dimension $k$ is a map
  $c\colon \Pi|_{U}(k)\to \R$. Its \emph{support} is the subset of $X$
  given by
  \begin{displaymath}
    |c| = \bigcup_{c(\sigma )\not = 0}\sigma.
  \end{displaymath}
  This weight is \emph{positive} if $c(\sigma )\ge 0$ for all
  $\sigma \in \Pi|_{U} (k)$.  For convenience, any weight $c$ on
  $\Pi|_U$ of dimension $k$ is extended to a function
  $c\colon \Pi|_{U}\rightarrow \R $ by setting $c(\sigma )=0$ for all
  $\sigma \in \Pi |_{U}(\ell)$ with $\ell \not = k$.

  A weight $c$ on $\Pi|_U$ of dimension $k$ is a
  \emph{Minkowski weight} if for each $\tau \in \Pi|_{U} (k-1)$,
  \begin{equation}
    \label{eq:36}
    \sum_{\mathclap{\substack{\sigma \in \Pi|_U (k)\\\sigma \succ \tau }}}c(\sigma
    )\, v_{\sigma \setminus \tau } =0.
  \end{equation}

  The set of weights on $\Pi|_U$ of dimension $k$, denoted
  $W_{k}(\Pi |_{U})$, is an Abelian group under the addition of
  functions. The subgroup of its Minkowski weights is denoted by
  $M_{k}(\Pi|_U)$, and the cone of those that are positive is denoted
  by $M_{k}^{+}(\Pi|_U)$.  For short, when $U=X$ we denote this
  Abelian group, subgroup and cone by $W_{k}(\Pi)$, $M_{k}(\Pi)$ and
  $M_{k}^{+}(\Pi)$, respectively.
\end{defn}

\begin{defn} A \emph{piecewise affine function} on the open subset $U$
  of the polyhedral space $X$ is a function $f\colon U \to \R$ for
  which there is a polyhedral complex $\Pi$ on $X$ such that for each
  $\sigma \in \Pi|_{U}$ the restriction $f| _{\sigma \cap U}$ is
  affine or equivalently, it is given by an element of
    $M_{\sigma}$. In this situation, we say that $f$ is
  \emph{defined} on $\Pi$.

For each $\sigma \in \Pi|_{U}$, we denote by 
  $f_{\sigma } $ a  linear function on $ N_{X }$  satisfying 
  \begin{displaymath}
    f|_{\sigma \cap U}=f_{\sigma } \circ \iota_{X,\sigma} |_{\sigma \cap U},
  \end{displaymath}
  Since the image of $\sigma$ in $N_{\sigma}$ spans a full-dimensional
  polyhedral cone, the restriction of $f_{\sigma }$ to
  $\iota_{X,\sigma}(N_{\sigma})$ does not depend on the choice of this
  linear function.

  We denote by $\PA(U)$ the Abelian group of piecewise affine functions
  on $U$, and by $\PA_{\Pi}(U)$ the subgroup of those piecewise affine
  functions that are defined on $\Pi$.   
\end{defn}

\begin{rem}
  \label{rem:1}
  Piecewise affine functions on polyhedral spaces are continuous,
  because they are continuous on the restriction to $U$ of each
  polyhedron of the polyhedral complex $\Pi$, and these polyhedra form
  a locally finite closed covering of the polyhedral space~$X$.
\end{rem}

We next define the basic operations on weights on polyhedral complexes
and study their interplay.

\begin{defn}[Restriction to open subsets]
  \label{def:20}
  Let $\Pi$ be a polyhedral complex on~$X$,  $V $ an open subset of
  $U$, and $c$ a $k$-dimensional weight on $\Pi|_{U}$. The
  \emph{restriction} of $c$ to $V $, denoted by $c|_{V }$, is  the
  $k$-dimensional weight on $\Pi|_{V }$ given by the restriction of
  this weight to the subset $\Pi|_{V }$ of $\Pi|_{U}$.
  \end{defn}

\begin{defn}[Pullback to subdivisions]
  \label{def:21}
  Let $\Pi ,\Pi'$ be polyhedral complexes on $X$ with $\Pi'\ge\Pi$ and
  $c$ a $k$-dimensional weight on $\Pi|_{U}$. The \emph{pullback} of
  $c$ to $\Pi '$, denoted by $c_{\Pi '}$, is the $k$-dimensional
  weight on $ \Pi' |_{U}$ defined, for $\sigma'\in \Pi'|_{U}$, by
  \begin{equation*}
    c_{\Pi '}(\sigma ')=
      \begin{cases}
        c(\sigma ) & \text{ if there is } \sigma \in \Pi|_U \text{
          with } \sigma  \supset \sigma' \text{ and }
        \dim(\sigma)=\dim(\sigma'), \\
        0 & \text{ else.}
      \end{cases}
  \end{equation*}
  
 \end{defn}

\begin{defn}[Product with piecewise affine functions]
  \label{def:intersection-product}
  Let $\Pi$ be a polyhedral complex on $X$,  $f$ a piecewise affine
  function on $U$ defined on $\Pi$, and
  $c$ a $k$-dimensional weight on $\Pi|_{U}$.
The  \emph{product} of $f$ and $c$, denoted by $f\cdot c$, is the
  $(k-1)$-dimensional weight on $\Pi|_{U}$ defined, for each
  $\tau \in \Pi|_{U} (k-1)$, by
  \begin{equation}\label{eq:1}
    (f\cdot c)(\tau )=- \sum_{\sigma \succ \tau }c(\sigma ) \, {f}_{\sigma}(v_{\sigma \setminus  \tau }),
  \end{equation}
  the sum being over the $k$-dimensional polyhedra
  $\sigma\in \Pi|_{U}$ having $\tau$ as a facet.
\end{defn}

Choosing any point $x\in \tau $, the formula in \eqref{eq:1} can be
alternatively written as
\begin{equation}
  \label{eq:2}
      (f\cdot c)(\tau )=\bigg(\sum_{\sigma \succ          \tau }c(\sigma )\bigg)
      \, f(x)
      - \sum_{\sigma \succ \tau }c(\sigma ) \, f_{\sigma }(\iota_{X,\sigma }(x)+v_{\sigma \setminus \tau }).
    \end{equation}

    \begin{prop}
      \label{prop:12}
      Let $\Pi$ be a polyhedral complex on $X$, $V $ an open subset of
      $U$, $\Pi'$ a subdivision of $\Pi$, $f$ a piecewise affine
      function on $U$ defined on $\Pi$, and $c$ a weight on
      $\Pi|_{U}$. Then
      \begin{displaymath}
        (c_{\Pi'})|_{V }= (c|_{V })_{\Pi'}, \quad  (f\cdot c)|_{V }=
        f|_{V }\cdot c|_{V } \and (f\cdot c)_{\Pi '} = f\cdot c_{\Pi '} .
      \end{displaymath}
    \end{prop}

    \begin{proof}
      The first equality, that is, the compatibility between the
      restriction to an open subset and the pullback to a
      subdivision, follows almost immediately from the definitions.
      Both $ (c_{\Pi'})|_{V }$ and $ (c|_{V })_{\Pi'}$ are
      $k$-dimensional weights on $\Pi'|_{V }$ and for each
      $\sigma'\in \Pi'|_{V }$, their possibly nonzero values  are
      respectively defined by
      \begin{enumerate}
      \item \label{item:10} $ (c_{\Pi'})|_{V }(\sigma')= c(\sigma)$ if there is
      $\sigma\in \Pi|_{U}$ with $\sigma\supset \sigma'$ and
      $\dim(\sigma)=\dim(\sigma')$, 
      \item \label{item:14} $ (c|_{V })_{\Pi'}(\sigma')= c(\wt\sigma)$ if there is
      $\wt\sigma\in \Pi|_{V }$ with $\wt\sigma\supset \sigma'$ and
      $\dim(\wt\sigma)=\dim(\sigma')$.
    \end{enumerate}
    Since $\sigma' \cap V \ne \emptyset$ and $\sigma' \supset \sigma$,
    we have that  $\sigma \in
    \Pi|_{V }$. Hence  $\sigma=\wt\sigma$ and both weights
    coincide, as stated. 

    The second equality is also direct from the definitions, since the
    product of a weight with a piecewise affine function is defined in
    local terms.
    
    Hence we turn to the third equality, giving the compatibility
    between the pullback to a subdivision and the product with a
    piecewise affine function. Both $(f\cdot c)_{\Pi '} $ and
    $ f\cdot c_{\Pi '} $ are $(k-1)$-dimensional weights on
    $\Pi'|_{U}$, and so it is enough to consider their values on the
    set of polyhedra $\Pi'|_{U}(k-1)$.

    Let $\tau'\in \Pi'|_{U}(k-1)$ and denote by $\tau$ the minimal
    polyhedron in $\Pi|_{U}$ containing $\tau'$. On the one hand, if
    $\dim(\tau)=k-1$ then, with notation as in Definition
    \ref{def:intersection-product},
    \begin{equation}
      \label{eq:31}
    (f\cdot c)_{\Pi '} (\tau')=       (f\cdot c) (\tau) =   -
    \sum_{\mathclap{\substack{\sigma \in \Pi|_U (k)\\\sigma \succ \tau
      }}}c(\sigma ) \, {f}_{\sigma }(v_{\sigma \setminus
      \tau }),
    \end{equation}
whereas if $\dim(\tau)\ge k$ then $     (f\cdot c)_{\Pi '}
(\tau')=0$. On the other hand, 
\begin{equation}
  \label{eq:39}
  (  f\cdot c_{\Pi '} )(\tau') =  - \sum_{\mathclap{\substack{\sigma' \in
      \Pi'|_U (k)\\\sigma' \succ \tau' }}}c_{\Pi'}(\sigma' ) \, {f}_{\sigma' }(v_{\sigma' \setminus
    \tau' }).
\end{equation}

When $\dim(\tau)\ge k+1$, for each $\sigma' \in \Pi'|_U (k)$ with
$\sigma' \succ \tau'$  the minimal polyhedron
$\sigma\in \Pi'|_{U}$ containing it also contains $\tau$ and so it has dimension at least
$k+1$. Hence $c_{\Pi'}(\sigma')= 0$ and the formula in~\eqref{eq:39}
implies that $ ( f\cdot c_{\Pi '} )(\tau') = 0$, proving the equality
in this case.

When $\dim(\tau)=k$, there are two polyhedra
$\sigma', \sigma'' \in \Pi'|_U (k)$ contained in $\tau$ and having
$\tau'$ as a facet.  We have that
\begin{displaymath}
c_{\Pi'} ( \sigma')=c_{\Pi'} (\sigma'') \and
v_{\sigma' \setminus \tau'} =  -v_{\sigma''\setminus \tau'},
\end{displaymath}
and also that $ {f}_{\sigma'}$ and ${f}_{\sigma''}$ coincide with $f_{\tau }$
  on
  ${\iota}_{X,\sigma}(N_{\sigma})={\iota}_{X,\sigma'}(N_{\sigma'})=\iota
  _{X,\tau }(N_{\tau })$. 
Hence, the contributions of these two polyhedra to the sum in
\eqref{eq:39} cancel. For any other polyhedron in $\Pi'|_U (k)$ having
$\tau'$ as a facet, the minimal polyhedron in $\Pi|_{U}$ containing it
has dimension greater than $k$ and so its value for the weight
$c_{\Pi'}$ is zero. Thus again $ ( f\cdot c_{\Pi '} )(\tau') = 0$ in
this case.

Finally suppose that $\dim(\tau)=k-1$. To each polyhedron
$\sigma' \in \Pi'|_{U}(k)$ contained in the $k$-dimensional skeleton
$\Sk_{k}(\Pi|_{U})$ and having $\tau$ as a facet, we associate the
minimal polyhedron $\sigma\in\Pi|_{U}$ containing it. This assignment
gives a bijection between this set of $k$-dimensional polyhedra of
$\Pi'|_{U}$ and that of polyhedra in $\Pi|_{U}(k)$ having $\tau$ as a
facet, and we have that
\begin{displaymath}
  c_{\Pi'}(\sigma') = c(\sigma) \and 
  v_{\sigma'\setminus \tau'}=v_{\sigma\setminus \tau}
\end{displaymath}
and also that $ {f}_{\sigma'}$ and ${f}_{\sigma''} $ coincide on
${\iota}_{X,\sigma}(N_{\sigma})={\iota}_{X,\sigma'}(N_{\sigma'})$. Hence
the sum in \eqref{eq:39} coincides with that in \eqref{eq:31}, since
the value of the weight $c_{\Pi'}$ at the polyhedra in $ \Pi'|_{U}(k)$
that are not contained in $\Sk_{k}(\Pi|_{U})$ is zero. Thus in this
case $ (f\cdot c)_{\Pi '} (\tau')= f\cdot c_{\Pi '} (\tau')$, which
concludes the proof.
\end{proof}

We next prove that the product of several piecewise affine functions
with a weight is commutative. This result is similar to
\cite[Proposition~3.7a]{AR}, and its proof is done in a similar way.
    
    \begin{prop}
      \label{prop:4}
      Let $\Pi$ be a polyhedral complex on $X$, $f,g$ piecewise affine
      functions on $U$ defined on $\Pi$, and $c$ a $k$-dimensional
      weight on $\Pi|_U$. Then
      \begin{displaymath}
 f\cdot (g\cdot c) = g\cdot (f\cdot c).
      \end{displaymath}
\end{prop}

\begin{proof}
  The proof is based on the following observation. Let $v_{1}$ and
  $v_{2}$ be two linearly independent unit vectors in a Euclidean
  space. Denote by $v_{1}^{\perp}$ and $v_{2}^{\perp}$ the unit
  vectors in the plane generated by $v_{1}$ and $v_{2}$, that are respectively
  orthogonal to $v_{1}$ and to $v_{2}$, and that both
  $v_{1},v_{1}^{\perp}$ and $v_{2}^{\perp},v_{2}$ have the same
  orientation as $v_{1},v_{2}$. Then for $a,b\in \R$,  the
   equations
  \begin{equation}
    \label{eq:40}
    v_{1}^{\perp}= a \, v_{1} + b\, v_{2} \and     v_{2}^{\perp}= a\, v_{2} + b\, v_{1}
  \end{equation}
  are equivalent.  Indeed, consider the reflection on the plane
  generated by $v_{1}$ and $v_{2}$ by the bisector of the angle
  between these two vectors. This reflection interchanges $v_{1}$ with
  $v_{2}$ and $v_{1}^{\perp}$ with $v_{2}^{\perp}$ and respects linear
  relations. Applying it to any of the two equations in \eqref{eq:40}
  gives the other one, proving that they are equivalent.

  Now let $\rho \in \Pi|_U(k-2)$. For each $\tau \in \Pi|_U(k-1)$ and
  $\sigma \in \Pi|_U(k)$ with $\rho \prec \tau \prec \sigma $, we
  denote by $\tau '$ the unique polyhedron different from $\tau $ that
  lies in $ \Pi|_U(k-1)$ and verifies that
  $\rho \prec \tau' \prec \sigma $.  The vectors
  $v_{\tau \setminus \rho }$, $v_{\tau' \setminus \rho }$,
  $v_{\sigma \setminus \tau }$ and $v_{\sigma \setminus \tau'}$
  satisfy the conditions of the vectors $v_{1}$, $v_{2}$,
  $v_{1}^{\perp}$ and $v_{2}^{\perp}$ in the previous discussion, and
  so there are real numbers $a_{\sigma ,\rho }$ and
  $b_{\sigma ,\rho }$ such that
  \begin{equation}\label{eq:32}
    v_{\sigma \setminus \tau }
    = a_{\sigma ,\rho } \, v_{\tau  \setminus \rho }+
    b_{\sigma ,\rho } \, v_{\tau ' \setminus \rho } \and 
    v_{\sigma \setminus \tau '}
    = a_{\sigma ,\rho } \, v_{\tau ' \setminus \rho }+
    b_{\sigma ,\rho } \, v_{\tau  \setminus \rho }.  
  \end{equation}
  We compute
  \begin{align*}
    (f\cdot(g\cdot c))(\rho )&= \sum_{\tau    \succ \rho }
\bigg(     \sum_{\sigma  \succ \tau }
    c(\sigma )\,  g_{\sigma }(v_{\sigma \setminus \tau }) \bigg) \,  f_{\tau
    }(v_{\tau \setminus \rho })\\
    &= \sum_{\tau
    \succ \rho}
    \sum_{\sigma  \succ \tau }
    c(\sigma )\,  g_{\sigma 
    }(a_{\sigma ,\rho }\, v_{\tau\setminus \rho }
    + b_{\sigma ,\rho }\, v_{\tau'\setminus \rho  })\,  f_{\tau
    }(v_{\tau \setminus \rho })\\
    &=\sum_{\sigma  \succ \rho}
    c(\sigma )\, a_{\sigma ,\rho }\bigg( \hspace{2mm}\sum_{\mathclap{\sigma \succ \tau
    \succ \rho }}  f_{\tau
    }(v_{\tau \setminus \rho })\,  g_{\tau 
    }(v_{\tau\setminus \rho })\bigg) \\
    &\phantom{AAA}+\sum_{\sigma  \succ \rho}
    c(\sigma )\, b_{\sigma ,\rho } \bigg( \hspace{2mm}\sum_{\mathclap{\sigma \succ \tau
    \succ \rho }}  f_{\tau
    }(v_{\tau \setminus \rho })\,  g_{\tau' 
    }(v_{\tau'\setminus \rho }) \bigg),
  \end{align*}
  where the indexes $\tau$ and $\sigma$ go over the sets of polyhedra
  $\Pi|_{U}(k-1)$ and $\Pi|_{U}(k)$, respectively.  The first equality
  is the definition of the product, the second follows from the first
  equation in \eqref{eq:32}, and the third comes from the linearity of
  $ g_{\sigma }$ combined with the fact that
  $ g_{\sigma }|_{N_{\tau} }= g_{\tau }$ and
  $ g_{\sigma }|_{N_{\tau'} }= g_{\tau' }$.

  Next, we interchange the roles of $\tau $ and $\tau '$ in the inner
  sum of the second term of the last expression and revert the
  previous argument, applying this time the linearity of
  $ f_{\sigma }$, the second equation in \eqref{eq:32} and again the
  definition of the product to obtain that
  \begin{align*}
(   f\cdot(g\cdot c))(\rho )&=\sum_{\sigma  \succ \rho}
    c(\sigma )\, a_{\sigma ,\rho }\bigg( \hspace{2mm} \sum_{\mathclap{\sigma \succ \tau
    \succ \rho }}  f_{\tau
    }(v_{\tau \setminus \rho })\,  g_{\tau 
    }(v_{\tau\setminus \rho })\bigg) \\
    &\phantom{AAA}+\sum_{\sigma  \succ \rho}
    c(\sigma )\, b_{\sigma ,\rho }\bigg( \hspace{2mm} \sum_{\mathclap{\sigma \succ \tau
    \succ \rho }}  f_{\tau'
    }(v_{\tau' \setminus \rho })\,  g_{\tau 
    }(v_{\tau\setminus \rho })\bigg) \\
        &= 
    \sum_{\tau
    \succ \rho }
              \sum_{\sigma  \succ \tau }
    c(\sigma )\,  f_{\sigma 
    }(a_{\sigma ,\rho }\, v_{\tau\setminus \rho }
    + b_{\sigma ,\rho }\, v_{\tau'\setminus \rho  })\,  g_{\tau
    }(v_{\tau \setminus \rho })\\
    &=\sum_{\tau  \succ \rho} \bigg ( 
    \sum_{\sigma \succ \tau} c(\sigma )\,   f_{\sigma 
    }(v_{\sigma \setminus \tau  }) \bigg) \,  g_{\tau 
    }(v_{\tau\setminus \rho })\\
    &=(g\cdot (f\cdot c))(\rho ),
  \end{align*}
  which proves the statement.
\end{proof}

In view of the definition of the product, we give the
following interpretation of the condition for an arbitrary weight to
be a Minkowski weight.

\begin{lem}
 \label{lemm:1} 
  Let $\Pi$ be polyhedral complex on $X$ and $c$ a weight on
  $\Pi|_{U}$.  Then $c$ is a Minkowski weight if and only if for
  every linear function $\ell\colon N_{X} \to \R$,
  \begin{displaymath}
    (\ell\circ \iota_{X})\cdot c = 0.
  \end{displaymath}
\end{lem}

\begin{proof}
For $\tau \in  \Pi|_U (k-1)$ we have that
  \begin{equation*}
    ((\ell\circ \iota_{X})\cdot c)(\tau )= - \sum_{\sigma \succ
      \tau }c(\sigma ) \, \ell(v_{\sigma \setminus \tau })
    = - \ell\Big(\sum_{\sigma \succ \tau }c(\sigma ) \, v_{\sigma \setminus \tau }\Big).
\end{equation*}
The statement follows from the fact that, for a vector $v \in N_{X}$,
the condition $v=0$ is equivalent to $\ell(v)=0$ for every linear
function $\ell$.
\end{proof}

\begin{prop}
  \label{prop:19}
  Let $\Pi$ be a polyhedral complex on $X$ and $c$ a $k$-dimensional
  Minkowski weight on $\Pi|_U$.
  \begin{enumerate}
  \item \label{item:15} If $V $ is an open subset of $U$, then
    $c|_{V }$  is  a $k$-dimensional Minkowski weight on $\Pi|_{V }$.  
  \item \label{item:11} If $\Pi '$ is a subdivision of $\Pi $, then
    $c_{\Pi'}$ is a $k$-dimensional Minkowski weight on $\Pi'|_{U}$.
  \item \label{item:12} If $f$ is a piecewise affine function on $U$ defined
    on~$\Pi$, then $ f\cdot c$ is a $(k-1)$-dimensional Minkowski
    weight on $\Pi|_{U}$.
     \end{enumerate}
\end{prop}

\begin{proof}
  The statement in \eqref{item:15} is direct from the definitions.

  To prove \eqref{item:11}, take a linear form $\ell $ on $N_{X}$ and
  set $h=\ell\circ \iota_{X}$ for short. Proposition~\ref{prop:12} and
  Lemma~\ref{lemm:1} imply that
    \begin{displaymath}
      h \cdot c_{\Pi'} = ( h \cdot c)_{\Pi'}=0.
    \end{displaymath}
    Since this holds for every $\ell$, Lemma \ref{lemm:1} implies that
    $ c_{\Pi'|_{U}}$ is a $k$-dimensional Minkowski weight on
    $\Pi'|_{U}$, proving the statement.

    For \eqref{item:12}, taking again a linear form $\ell $ on $N_{X}$
    and setting $h=\ell \circ \iota_{X}$, we get from Proposition
    \ref{prop:12} and Lemma \ref{lemm:1} that
  \begin{displaymath}
    h\cdot (f\cdot c)=f\cdot (h\cdot c)=0.
  \end{displaymath}
  Since this holds for every $\ell$, Lemma \ref{lemm:1} again implies
  that $f\cdot c$ is a $(k-1)$-dimensional Minkowski weight
  on~$\Pi|_{U}$, as stated.
\end{proof}

Hence the operations of restriction to an open subset, pullback to a
subdivision and product with piecewise affine functions induce the
families of homomorphisms that in the notation of Proposition
\ref{prop:19}, write down, for each $k\in\Z_{\ge 0}$, as
\begin{align}
  \label{eq:35}
&    M_{k}(\Pi|_{U})   \longrightarrow M_{k}(\Pi|_{V }), \quad
  c\longmapsto c|_{V }, \\
  \label{eq:37}
  & M_{k}(\Pi|_{U})   \longrightarrow M_{k}(\Pi'|_{U}),  \quad
  c\longmapsto c_{\Pi'},\\ 
\label{eq:38}
  & \PA_{\Pi}(U)\times M_{k}(\Pi|_{U})   \longrightarrow M_{k-1}(\Pi|_{U}), \quad
(f,  c)\longmapsto f\cdot c.
\end{align}

\begin{defn}\label{def:14}
For $k\in \Z_{\ge 0}$, the space of  \emph{Minkowski cycles} on $U$ of
dimension $k$ is
  defined as the direct limit
  \begin{displaymath}
    Z_{k}(U)=\varinjlim_{\Pi }M_{k}(\Pi|_U),
  \end{displaymath}
  taken over the directed set of polyhedral complexes on $X$ ordered
  by subdivision, and where the map corresponding to each pair
  $\Pi' \ge \Pi$ is the pullback homomorphism in
  \eqref{eq:37}.

  Given a Minkowski weight $c \in M_{k}(\Pi|_{U})$, we denote by
  $[c]\in Z_{k}(U)$ the associated Minkowski cycle. Conversely, given
  a Minkowski cycle $\gamma \in Z_{k}(U)$ and a polyhedral complex
  $\Pi $ on $X$, we say that $\gamma$ is \emph{defined on $\Pi $} if
  there is $c\in M_{k}(\Pi|_{U})$ such that $\gamma=[c]$. The
  \emph{support} of a Minkowski cycle $\gamma $ is the support of any
  Minkowski weight representing it.
  
  The Minkowski cycle $\gamma$ is \emph{positive} if it can be
  represented by a positive Minkowski weight on $\Pi|_{U}$. We denote
  by $Z_{k}^{+}(U)$ the cone of positive Minkowski cycles on $U$ of
  dimension $k$.
\end{defn}

The compatibility between the pull-back of Minkowski weights to
subdivisions on the one hand and the restriction to open subsets and
the product with piecewise affine functions on the other
(Proposition~\ref{prop:12}) allows to define the corresponding
operations for Minkowski cycles. Namely, from \eqref{eq:35} and
\eqref{eq:38} we derive the families of homomorphisms given, for
$k\in \Z_{\ge 0}$, by
\begin{align*}
&    Z_{k}(U)   \longrightarrow Z_{k}({V }), \quad
                  \gamma\longmapsto \gamma|_{V }:=[c|_{V }] , \\
  \label{eq:21}
  & \PA(U) \times Z_{k}({U})   \longrightarrow Z_{k-1}({U}), \quad
(f,  \gamma)\longmapsto f\cdot \gamma :=[f\cdot c]
\end{align*}
for any polyhedral complex $\Pi$ on $X$ and
$c\in M_{k}(\Pi|_{U})$ with $\gamma=[c]$.

  We next introduce the notions of balancing condition and balanced
  polyhedral space.

  \begin{defn}\label{def:balanced}
    Suppose that the Euclidean polyhedral space $X$ has pure dimension
    $n$. A \emph{balancing condition} on $U$ is an $n$-dimensional
    Minkowski cycle $\beta$ on $U$ that is represented by an
    $n$-dimensional Minkowski weight $b$ on $\Pi|_{U}$ for a
    polyhedral complex $\Pi$ on $X$, and such that $ b(\sigma )>0$ for
    all $\sigma \in \Pi |_{U}(n)$.  The pair $(U,\beta )$ is called a
    \emph{balanced open subset}.  When $U=X$, the pair $(X,\beta)$ is
    called a \emph{balanced polyhedral space}.  A \emph{polyhedral complex}
    $\Pi$ on a balanced open subset $(U,\beta)$ is a polyhedral complex
    such that the balancing condition $\beta$ is defined on $\Pi$.
    
    The Euclidean polyhedral space $X$ is \emph{balanceable} if it
    admits a balancing condition on $X$. It is \emph{locally
      balanceable} if there is an open covering $X=\bigcup_{i} U_{i}$
    admitting a balancing condition on each $U_{i}$.
\end{defn}

We will usually denote a balanced open subset by its underlying subset
$U$ and, in this case, we will denote the corresponding balancing
condition by $\beta_{U}$.

\begin{exmpl}\label{exm:2}
  For $n\in \Z_{\ge 0}$, the vector space $\R^{n}$ can be given a
  structure of a balanced polyhedral space by considering the
  polyhedral complex $\Pi$ consisting of the single polyhedron
  $\sigma =\R^{n}$ embedded in the Euclidean vector space
  $\R^{n+1}$
  through the map $x\mapsto(x,1)$, and the Minkowski weight defined by
  $\beta_{\R^{n}}(\sigma)=1$. 
\end{exmpl}

  \begin{rem}
    \label{rem:101}
    As in Example \ref{exm:2}, in many situations the space $N$ will
    consist of the Euclidean space $\R^{n+1}$ and $H$ of the
    hyperplane $(x_{n+1}=1)$. In those situations we will only
    explicit the hyperplane $H$, whereas the space $N$ will be tacitly
    assumed.
  \end{rem}

\begin{exmpl}\label{exm:6}
  Let $X$ be a Euclidean polyhedral space, $U\subset X$ an open
  subset and $\gamma \in Z^{+}_{k}(U)$ a positive Minkowski
  cycle. Then the support $|\gamma |$ has an induced structure of
  Euclidean polyhedral space of pure dimension $k$ with a balancing
  condition on $|\gamma |\cap U$, given by the restriction of the
  cycle $\gamma$ to its support. In particular, if $U=X$ then
  $|\gamma |$ is balanced.
\end{exmpl}

\begin{exmpl}\label{exm:13}
  Not every Euclidean polyhedral space is balanceable. For instance,
  the ray $\R_{\ge 0}$ with its standard structure of a Euclidean
  polyhedral space given by the polyhedra $\{0\}$ and $\R_{\ge 0}$, is
  not balanceable.  Indeed, let $\Pi $ be a polyhedral complex on
  $\R_{\ge0}$ and set $\tau =\{0\}\in \Pi (0)$. There is a unique
  polyhedron $\sigma \in \Pi (1)$ with $\tau \prec \sigma $ and, for
  any $c\in M_{1}(\Pi )$, the condition \eqref{eq:36} boils down to
 \begin{displaymath}
    c(\sigma ) v_{\sigma \setminus \tau }=0.
  \end{displaymath}
Hence $c(\sigma)=0$ and so  no balancing condition can be
defined on $\Pi$.

More generally, any convex polyhedron $\Delta \subset \R^{n}$ of
maximal dimension is a Euclidean polyhedral space with the structure
induced from that of $\R^{n}$ in Example~\ref{exm:2} but, unless
$\Delta= \R^{n}$, it is not balanceable.
\end{exmpl}

\begin{exmpl}
  Let $X$ be the Euclidean polyhedral space depicted in Figure
  \ref{fig:locally_bal}, where $ABC$ is an equilateral triangle, $L$ is
  the bisector of the angle opposed to $A$ and $M$ is the bisector of
  the angle opposed to $B$. Then $X$ is balanceable if and only if $N$
  is the bisector of the angle opposed to $C$, while it is locally
  balanceable if an only if $N$ is contained in the interior of the
  angle opposed to $C$.

  \begin{figure}[ht]
  \centering
  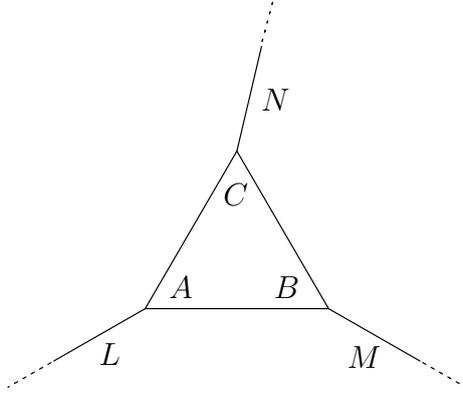
  \caption{A locally balanceable space}
  \label{fig:locally_bal}
  \end{figure}
\end{exmpl}

We end this section discussing the change of Euclidean structure.

\begin{defn}\label{def:9}
  Let $X'$ be another Euclidean polyhedral space that has the same
  underlying quasi-embedded polyhedral space of $X$ and a possibly
  different Euclidean structure. Let $\Pi $ be a polyhedral complex on
  $X$ and $\Pi '$ the corresponding polyhedral complex on $X'$. For
  $k\in \Z_{\ge 0}$, the \emph{change of Euclidean structure} on the
  open subset $U$ is the morphism
\begin{math}
\varphi_{X',X}\colon W_{k}(\Pi |_{U})\to W_{k}(\Pi' |_{U}) 
\end{math}
given, for $c\in  W_{k}(\Pi |_{U})$ and $\sigma\in \Pi|_{U}(k)$, by
\begin{displaymath}
  \varphi_{X',X}(c)(\sigma)={\frac{\vol_{X'}(\sigma
      )}{\vol_{X}(\sigma)}}\, c(\sigma ),
\end{displaymath}
where $\vol_{X'}(\sigma)/\vol_{X}(\sigma)$ denotes the ratio between
the Haar measures on the affine space $N_{\sigma}$ induced by the
Euclidean structures of $X'$ and $X$.
\end{defn}

\begin{rem}\label{rem:6}
  It follows directly form the definition that for three Euclidean
  structures $X$, $X'$ and $X''$ on the same polyhedral space, the
  change of Euclidean structure maps induced by a polyhedral complex on $X$ verify that
  \begin{displaymath}
    \varphi_{X'',X}=\varphi_{X'',X'}\circ \varphi_{X',X}.
  \end{displaymath}
\end{rem}

\begin{prop}\label{prop:5}
  With notation as in Definition \ref{def:9}, let
  $c\in M_{k}(\Pi |_{U})$ and $f\in \PA_{\Pi }(U)$. Then
  \begin{displaymath}
    \varphi_{X',X}(c)\in M_{k}(\Pi '|_{U}) \and \varphi_{X',X}(f\cdot
    c)=f\cdot \varphi_{X',X}(c).
  \end{displaymath}
\end{prop}

\begin{proof}
  We denote with a prime the objects relative to $X'$. For instance,
  for $\tau \in \Pi |_{U}(k-1)$ and $\sigma \in \Pi |_{U}(k)$ with
  $\tau \prec \sigma $, the symbol $v'_{\sigma \setminus \tau }$
  denotes the unit normal vector to $\tau $ in the direction of
  $\sigma $ for the Euclidean structure of $X'$, while
  $v_{\sigma \setminus \tau }$ denotes that for $X$.

  We decompose this vector as
  \begin{equation*}
    v'_{\sigma \setminus \tau }=\langle v'_{\sigma \setminus \tau
    },v_{\sigma \setminus \tau }\rangle \, v_{\sigma \setminus \tau }
    +w_{\sigma \setminus \tau },
  \end{equation*}
  where the scalar product is computed in $X$ and
  $w_{\sigma \setminus \tau }$ is a vector parallel to  
  $ \iota_{X,\tau }(H_{\tau })$. The ratio between the Haar
  measures induced on the affine spaces $H_{\tau} $ and $H_{\sigma} $
  are related by
  \begin{equation*}
    \frac{\vol_{X'}(\tau )}{\vol_{X}(\tau)} =
\langle v'_{\sigma \setminus \tau
    },v_{\sigma \setminus \tau }\rangle \, 
    \frac{\vol_{X'}(\sigma)}{\vol_{X}(\sigma)}  .
  \end{equation*}
  To prove that $\varphi_{X',X}(c)$ is a Minkowski weight we
  compute, for $\tau\in \Pi|_{U}(k-1)$,
  \begin{multline}\label{eq:53}
    \sum_{\sigma \succ \tau } \frac{\vol_{X'}(\sigma
      )}{\vol_{X}(\sigma )}\, c(\sigma )\, 
    v'_{\sigma \setminus \tau
    }=
    \sum_{\sigma \succ \tau } \frac{\vol_{X'}(\sigma
      )}{\vol_{X}(\sigma )}\, c(\sigma )\, (
      \langle v'_{\sigma \setminus \tau
    },v_{\sigma \setminus \tau }\rangle v_{\sigma \setminus \tau }
    +w_{\sigma \setminus \tau })\\ =
  \frac{\vol_{X'}(\tau 
      )}{\vol_{X}(\tau )} \Big( \sum_{\sigma \succ \tau }c(\sigma ) \, 
    v_{\sigma \setminus \tau }\Big) +w_{\tau } =w_{\tau }
  \end{multline}
  for some vector $w_{\tau }$ parallel to
  $\iota_{X,\tau }(H_{\tau })$. Since $w_{\tau}$ is a linear
  combination of the vectors $v'_{\sigma \setminus \tau}$,
  $\sigma\succ \tau$, and these vectors are orthogonal to the affine
  subspace $\iota_{X,\tau }(H_{\tau })$ with respect to the Euclidean
  structure of $X'$, we deduce that $w_{\tau }=0$, proving the first
  statement.

  To prove the second statement, we compute, using the notation
  in~\eqref{eq:53} and the fact that $w_{\tau }=0$,
\begin{multline*}
  (f\cdot \varphi_{X',X}(c))(\tau )= -
  \sum_{\sigma \succ \tau } \frac{\vol_{X'}(\sigma
      )}{\vol_{X}(\sigma )}\, c(\sigma ) \, 
f_{\sigma }(v'_{\sigma \setminus \tau
    } ) \\= - \frac{\vol_{X'}(\tau 
      )}{\vol_{X}(\tau )}\Big(\sum_{\sigma \succ \tau }c(\sigma) \,
     f_{\sigma}(v_{\sigma \setminus \tau }) \Big) -  f_{\tau }(w_{\tau })=
    \varphi_{X',X}(f\cdot c)(\tau ).
\end{multline*}
\end{proof}

Thanks to this result, we can consider the change of Euclidean
  structure map from Definition \ref{def:9} at the level of Minkowski
  weights and Minkowski cycles:
\begin{equation}
  \label{eq:55}
  \varphi_{X',X} \colon M_{k}(\Pi |_{U})\to M_{k}(\Pi' |_{U}) \and
\varphi_{X',X}\colon Z_{k}(U)\to Z_{k}(U').
\end{equation}
They also satisfy the composition property of Remark  \ref{rem:6}.

\section{Concave piecewise affine functions on polyhedral spaces}
\label{sec:conc-conv-funct}

In this section we introduce and discuss the different notions of
concavity for a piecewise affine function on an open subset of a
polyhedral space.  We denote by $X$ a quasi-embedded polyhedral space
and $U$ an open subset of it.

First we recall the definition of a concave function defined on a
vector space and possibly taking infinite values.

\begin{defn}\label{def:7}
  A function $f\colon \R^{n}\to \R\cup \{\pm \infty\}$ is 
  \emph{concave} if its \emph{hypograph}
  \begin{displaymath}
    \hypo(f)=\left\{(x,z)\in \R^{n}\times \R\,\mid\, z \le f(x)\right\} 
  \end{displaymath}
  is a convex subset of $\R^{n }\times \R$.
\end{defn}

For a function $f\colon \R^{n}\to \R\cup \{-\infty\}$, being concave
in the sense of Definition \ref{def:7} is equivalent to the usual
condition that, for all $x,x'\in \R^{n}$ and
$\nu_{1}, \nu_{2}\in \R_{\ge 0}$ with $\nu_{1}+\nu_{2}=1$,
\begin{displaymath}
  f(\nu_{1}\, x+\nu_{2}\, x') \ge   \nu_{1}\, f(x) + \nu_{2}\, f(x').
\end{displaymath}
However, this definition cannot be directly extended to functions on
the open subset $U$, as we do not dispose of a notion of convexity for
subsets of a polyhedral space. Instead, we can use the quasi-embedding
to pullback the notion of concavity for functions on the vector space
$N_{X}$.

\begin{defn}\label{def:6} Let $f$ be a piecewise
  affine function on $U$. Then $f$ is \emph{strongly concave} if there
  is a concave function $f'\colon H_{X}\to \R\cup \{\pm \infty\}$ as
  in Definition \ref{def:7}, such that
  \begin{displaymath}
    f=(f'\circ \iota_{X}) |_{U}.
  \end{displaymath}
  The piecewise affine function $f$ is \emph{locally strongly concave}
  if there is an open covering $U=\bigcup_{i} U_{i}$ such that
  $f|_{U_{i}}$ is strongly concave for every $i$.
\end{defn}

Similar definitions have been proposed in the context of tropical
geometry, as in \cite[Section 1.5]{Rau}.  However, neither of them is
completely satisfactory. On the one hand, the notion of being strongly
concave is not local as shown by the next example but, as we will see
in Proposition \ref{prop:9}, can be extended to a notion that is
stable under the operation of taking the infimum. On the other hand,
the notion of being locally strongly concave is local, but as also
shown in the next example, cannot be extended to a notion that is
stable under the operation of taking the infimum.

\begin{exmpl}
  \label{exm:8}
  Let $X$ be the $1$-dimensional polyhedral space shown in
  Figure~\ref{fig:pol-comp} and consisting of the four polyhedra
  $\sigma _{i}=\R_{\ge 0}\times \{i\}$, $i=1,\dots,4$, glued together
  by the points $(0,i)$. Set $H=\R^{2}$ and let $\iota \colon X\to H$
  be the quasi-embedding defined, for $x\in \R_{\ge0}$ and
  $ i\in \{1,2,3,4\}$, by
  \begin{displaymath}
    \iota(x,i)=
    \begin{cases}
      (0,x)&\text{ if }i=1,\\
      (0,-x)&\text{ if }i=2,\\
      (x,x)&\text{ if }i=3,\\
      (-x,x)&\text{ if }i=4.
    \end{cases}
  \end{displaymath}
  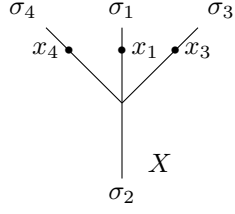
\begin{figure}[h]
    \begin{center}
      \begin{tikzpicture}[scale=1]
        \draw (0,0) -- (0,1);
        \draw (0,0) -- (1,1);
        \draw (0,0)--(-1,1);
        \draw (0,0)--(0,-1); 
        \draw (0.5,-0.8) node{$X$};
        \draw (1,1) node[above right]{$\sigma_3$};
        \draw (-1,1) node[above left]{$\sigma_4$};
        \draw (0,-1) node[below]{$\sigma_2$};
        \draw (0,1) node[above]{$\sigma_1$};
        \draw (0.7, 0.7) node{\tiny{$\bullet$}};
        \draw (0.7, 0.7) node[right]{$x_{3}$};
        \draw (-0.7,0.7) node{\tiny{$\bullet$}};
        \draw (-0.7,0.7) node[left]{$x_4$};
        \draw (0.,0.7) node{\tiny{$\bullet$}};
        \draw (0.,0.7) node[right]{$x_1$};
      \end{tikzpicture}
    \end{center}
    \caption{A polyhedral complex}\label{fig:pol-comp}
  \end{figure}

  For each integer $n\ge 2$ consider the piecewise affine function
  $f_{n}\colon X\to \R$ given by
  \begin{displaymath}
    f_{n}(x,i)=
    \begin{cases}
\displaystyle
      \min\Big(0,\frac{1}{n}-x\Big)&\text{ if } i=1,\\
      0&\text{ otherwise}.
    \end{cases}
  \end{displaymath}
  It is not strongly concave because for the points 
  $x_{1}=\iota (1,1)$, $x_{3}=\iota (1,3)$ and $x_{4}=\iota (1,4)$
  shown in Figure \ref{fig:pol-comp} we have that
  $ \iota (x_{1})=\frac{1}{2} \, x_{3}+ \frac{1}{2}\, x_{4}$ and
  \begin{displaymath}
    f_{n}(1,1)=\frac{1}{n}-1 <  0=\frac{1}{2} \, f_{n}(1,3)+ \frac{1}{2} \, f_{n}(1,4). 
  \end{displaymath}
  Hence $f_{n}$ is not the pullback to $X$ of a concave function on
  $H$.  On the other hand, this piecewise affine function is locally
  strongly concave, as it can be seen by considering the covering
  $X=U_{n}\cup V_{n}$ with
  \begin{displaymath}
    U_{n}=\Big\{(x,i)\mid i\not = 1 \text{ or } x<\frac{2}{3\, n} \Big\} \and
    V_{n}=\Big\{(x,i)\mid i = 1 \text{ and } x>\frac{1}{3\, n} \Big\}.
  \end{displaymath}
  This shows that the notions of being strongly concave and of being
  locally strongly concave do not agree, and the notion of being
  strongly concave is not local.

Moreover, consider the piecewise affine function $f \colon X\to \R$ defined by
\begin{equation}
  \label{eq:22}
    f(x,i)=
        \begin{cases}
      -x&\text{ if } i=1,\\
      0&\text{ otherwise}.
    \end{cases}
  \end{equation}
  It is the infimum of the family $\{f_{n}\}_{n\ge 2}$, but it is not
  locally strongly concave. Indeed, for any $\varepsilon>0$ we have
  that
  $\iota (\varepsilon,1)= \frac{1}{2}\, \iota (\varepsilon,3)+
  \frac{1}{2}\, \iota(\varepsilon,4)$ and
  \begin{displaymath}
    f(\varepsilon,1)= -\varepsilon < 0 = \frac{1}{2}\,
    f(\varepsilon,3)+ \frac{1}{2} \, f(\varepsilon,4),
  \end{displaymath}
  and so the restriction of $f$ to any open neighborhood of the point
  $(0,1)$ is not the pullback of a concave function on $H$.  Thus the
  notion of being locally strongly concave is not closed under the
  operation of taking the infimum.
\end{exmpl}

In Definition \ref{def:6}, the concave function $f'$ on $H_{X}$ might
in principle take the values $\pm\infty$, even though the piecewise
affine function $f$ on $U$ takes only values in $\R$.  The next
example shows that including these infinite values does make a
difference in the notion of strong concavity, with respect to a
definition where these values for $f'$ are not allowed.
  
\begin{exmpl}
  \label{exm:11}
  Let $X$ be the $1$-dimensional polyhedral space consisting of the
  union of the two real lines $\R\times \{0\}$ and $ \R\times \{1\}$
  of $H=\R^{2}$.  Consider the piecewise affine function
  $f\colon X\to \R\cup\{\pm\infty\}$ defined, for $x\in \R$, by
  $f(x,0)= x$ and $f(x,1)=-x$, and let $f'\colon \R^{2}\to \R$ be the
  function defined, for $(x,y)\in \R^{2}$, by
  \begin{displaymath}
    f'(x,y)=
    \begin{cases}
      +\infty&\text{ if } 0<y<1,\\
      x&\text{ if } y=0,\\
      -x & \text{ if } y=1,\\
      -\infty&\text{ if } y<0 \text{ or } y>1.      
    \end{cases}
  \end{displaymath}
  It is concave according to Definition \ref{def:7} and since
  $f=f'\circ \iota$, we deduce that $f$ is strongly concave in the
  sense of Definition \ref{def:6}.  Moreover, any concave function
  $f''\colon \R^{2}\to \R \cup\{\pm\infty \}$ with $f=f''\circ \iota$
  should satisfy that
  \begin{displaymath}
    f''\Big(0,\frac{1}{2} \Big) \ge \frac{1}{2} \, f(x,0)+
    \frac{1}{2}\, f(-x,1)=x
  \end{displaymath}
  for all $x\in \R$. Hence, such a concave function necessarily takes
  the value $ +\infty$ at the point $\big(0,\frac{1}{2}\big)$.
\end{exmpl}

In view of the issues concerning these notions of concavity, we
propose two other ones based on the idea of preserving the
positivity of Minkowski cycles. For the rest of this section, we
  assume that $X$ is a Euclidean polyhedral space and that $U$ is an
  open subset of it.

\begin{defn}\label{def:3}
  Let $f$ be a piecewise affine function on $U$.
  \begin{enumerate}
  \item The piecewise affine function $f$ is \emph{concave} if for
    every positive Minkowski cycle $\gamma$ on an open subset $V$ of
    $U$, the product $f\cdot \gamma$ is also a positive Minkowski
    cycle on $V$.
  \item When $U$ is balanced with balancing condition $\beta_{U}$, then
    $f$ is \emph{weakly concave} if the product $f\cdot \beta_{U}$ is a
    positive Minkowski cycle on $U$.
  \end{enumerate}  
\end{defn}

The sets of strongly concave, of concave, and of weakly concave piecewise
affine functions on $U$ are respectively denoted by
\begin{displaymath}
  \SCPA(U), \quad \CPA(U) , \quad\WCPA(U).
\end{displaymath}
These are cones in $\PA(U)$.  Similarly, for a polyhedral complex
$\Pi$ on $X$, the cones of strongly concave, of concave, and of weakly
concave piecewise affine functions on~$U$ that are defined on $\Pi$
are respectively denoted by
\begin{displaymath}
  \SCPA_{\Pi}(U), \quad \CPA_{\Pi}(U) , \quad \WCPA_{\Pi}(U).
\end{displaymath}
When we want to stress the dependency of the class of weakly
  concave piecewise affine functions on the balancing condition, we
  denote its corresponding spaces by $\WCPA(U,\beta_{U} )$ and
  $\WCPA_{\Pi }(U,\beta_{U} )$.

Clearly, the notion of strongly concave piecewise affine function does
not depend on the choice of the Euclidean structure of $X$ .  The next
result shows the dependency of the other notions of concavity
with respect to changes on the Euclidean structure.

\begin{prop}\label{prop:21}
      Let $X'$ be a Euclidean polyhedral space with the same
      quasi-embedded polyhedral space of $X$, let $U'$ be the open
      subset of $X'$ corresponding to $U$ and
      $\varphi_{X',X}\colon Z_{k}(U)\to Z_{k}(U')$, $k\in \Z_{\ge0}$,
      the change of Euclidean structure maps as in \eqref{eq:55}. Then
      \begin{displaymath}
    \CPA(U)=\CPA(U')\and \WCPA(U,\beta_{U} )=\WCPA(U',\varphi_{X',X}(\beta_{U}) ).
  \end{displaymath}
\end{prop}

\begin{proof}
  This is a direct consequence of Proposition \ref{prop:5}.
\end{proof}

We have seen in Example \ref{exm:8} that, for a piecewise affine
function on an open subset of a polyhedral space, the notion of being strongly
concave is not local. By contrast, the notions of being
concave and of being weakly concave are local.

\begin{prop} \label{prop:14} Let $f$ be a piecewise affine function on
  $U$, and $U=\bigcup_{i}U_{i}$ an open covering.
  \begin{enumerate}
  \item $f$ is
  concave if and only if $f|_{U_{i}}$ is
  concave for all $i$.
\item When $U$ is balanced, $f$ is weakly concave if and only if
  $f|_{U_{i}}$ is weakly concave for all $i$.
  \end{enumerate}  
\end{prop}

\begin{proof}
  This follows directly from the definitions and the compatibility
  between the restriction to open subsets and the product of piecewise
  affine functions with Minkowski weights (Proposition \ref{prop:12}).
\end{proof}

The next result motivates our terminology for the different notions
of concavity.

\begin{prop}
  \label{prop:11}
  Let $f$ be a piecewise affine function on  $U$.
  \begin{enumerate}
  \item \label{item:8} If  $f$ is locally strongly
    concave, then $f$ is concave.
  \item \label{item:9} If $U$ is balanced and $f$ is concave, then $f$
    is weakly concave.
  \end{enumerate}
\end{prop}

\begin{proof}
  To prove \eqref{item:8}, let $\Pi$ be a polyhedral complex on $X$
  where $f$ is defined, $V\subset U $ an open subset,
  $k\in \Z_{\ge 0}$, and $c \in M_{k}^{+}(\Pi|_V)$. Let
  $\tau \in \Pi|_V(k-1)$ and set
  \begin{equation}\label{eq:4}
    S= \hspace{2mm}\sum_{\mathclap{\substack{\sigma  \in \Pi|_V(k)\\ \sigma \succ \tau    }}}c(\sigma   ).
  \end{equation}
If $S=0$ then $c(\sigma )=0$ for all $\sigma \in \Pi|_V(k)$ with
$\sigma\succ \tau$
  because $c$ is positive. Therefore $(f\cdot c)(\tau )=0 $ in this
  case.
  
  Else $S> 0$, and pick then $x\in \relint(\tau )\cap V$. Since $f$ is
  locally strongly concave, we can choose an open neighborhood $V'$ in
  $ V$ of the point $x$, a convex subset $C$ of $N_{X}$ and a concave
  function $f'\colon C \to \R$ such that
  \begin{displaymath}
f|_{V'}=f'\circ \iota_{X} .    
  \end{displaymath}
  Choose also $\varepsilon >0$ so that
  $\iota_{X}(x)+\varepsilon \, v_{\sigma \setminus \tau }\in
  \iota_{X}(\sigma \cap V')$ for all $\sigma \in \Pi|_V(k)$
  with~$\sigma\succ \tau$.  By \eqref{eq:4} we have that
  $ \sum_{\sigma \succ \tau } \frac{c(\sigma )}{S}=1$ and, since $c$
  is a Minkowski weight, we also have that
  $ \sum_{\sigma \succ \tau } c(\sigma )\,v_{\sigma \setminus \tau
  }=0$.  Therefore, the concavity of $f'$ implies that
  \begin{equation}
    \label{eq:50}
    f'(\iota_{X}(x))\ge
    \sum_{\sigma \succ \tau  }
    \frac{c(\sigma  )}{S} f'(\iota_{X}(x)+\varepsilon \, v_{\sigma \setminus \tau  }).
  \end{equation}
  From the formula in \eqref{eq:2} and the fact that $f$ is affine on
  each polyhedron $\sigma \in \Pi|_V$ we get that
  \begin{equation*}
    (f\cdot c)(\tau  )=\frac{S}{\varepsilon } 
    \bigg( f'(\iota_{X}(x))-
    \sum_{\sigma \succ \tau  }
    \frac{c(\sigma  )}{S} f'(\iota_{X}(x)+\varepsilon \, v_{\sigma \setminus \tau  })\bigg),
\end{equation*}
and so \eqref{eq:50} implies that $(f\cdot c)(\tau )\ge 0$. Hence
$f\cdot c$ is positive, and thus $f$ is concave.

The statement in \eqref{item:9} is immediate from the definitions,
because the balancing condition $\beta_{U}$ is a positive Minkowski cycle.
\end{proof}

The next result shows that, for a given piecewise affine function, the
condition of being concave can be checked in any polyhedral complex
where it is defined.
  
\begin{prop}
  \label{prop:13} Let $f$ be a piecewise affine function on $U$ defined on a
  polyhedral complex $\Pi$. The following conditions are equivalent:
  \begin{enumerate}
  \item \label{item:3}  $f$ is concave,
  \item\label{item:4} for every open subset $V\subset U$, every $k\in \Z_{\ge 0}$
  and every $c\in M_{k}^{+}(\Pi |_{V})$, we have that
  $f\cdot c \in M_{k-1}^{+}(\Pi |_{V})$.
\end{enumerate}
\end{prop}

\begin{proof}
  Suppose that $f$ is concave and let $c\in M_{k}^{+}(\Pi
  |_{V})$. Since $ [f\cdot c]=f\cdot [c]$ is a positive Minkowski
  cycle, we deduce that $f\cdot c$ is a positive Minkowski weight, and
  so $f$ satisfies the condition \eqref{item:4}.

  To prove the converse, suppose that $f$ satisfies the condition
  \eqref{item:4} and choose an open subset $V\subset U$,
  $k\in \Z_{\ge 0}$ and $\gamma\in Z_{k}^{+}(V)$.  Let $\Pi '$ be a
  subdivision of $\Pi $ and $c'\in M_{k}^{+}(\Pi '|_V)$ a positive
  Minkowski weight representing $\gamma$.

  Let $\tau \in \Pi '|_{V}(k-1)$ and set $\mu$ for the minimal
  polyhedron of $\Pi|_{V} $ containing~$\tau $.  For each
  $\sigma \in \Pi '|_{V}(k)$ with $\sigma \succ \tau $, let
  $\lambda (\sigma )$ be the minimal polyhedron of $\Pi|_{V} $
  containing $\sigma $.  We decompose the unit normal vector to $\tau$
  in the direction of $\sigma$ as
  \begin{equation}
    \label{eq:41}
      v_{\sigma \setminus \tau }=v_{\sigma ,\mu }+v_{\sigma ,\mu ^{\perp}}
    \end{equation}
    with $v_{\sigma ,\mu }$ parallel to the affine subspace
    $\iota_{X}(H_{\mu} )$ and $v_{\sigma ,\mu ^{\perp}}$ orthogonal to
    it.  Then
    \begin{equation*}
      0=\sum_{\sigma \succ \tau }c'(\sigma ) \, v_{\sigma \setminus \tau }=
      \sum_{\sigma \succ \tau }c'(\sigma ) \, v_{\sigma ,\mu }+\sum_{\sigma \succ
        \tau }c'(\sigma ) \, v_{\sigma ,\mu ^{\perp}}. 
    \end{equation*}
    Since this decomposition is orthogonal, we deduce that
    \begin{equation}
      \label{eq:28}
       \sum_{\sigma \succ \tau }c'(\sigma ) \, v_{\sigma ,\mu
    }=\sum_{\sigma \succ \tau }c'(\sigma ) \, v_{\sigma ,\mu
      ^{\perp}}=0.
    \end{equation}
    Put $l=\dim(\mu )$, and for each $\sigma$ as above write
    \begin{equation}
      \label{eq:43}
      v_{\sigma ,\mu ^{\perp}}=\hspace{2mm}\sum_{\mathclap{\mu \prec \nu  \prec \lambda (\sigma )}} a_{\sigma, \nu 
    }\, v_{\nu  \setminus \mu }
  \end{equation}
  with $ a_{\sigma, \nu } \in \R_{\ge0}$ and where the sum goes over the
  polyhedra $\nu\in \Pi|_{V} (l+1)$ with
  $ \mu \prec \nu \prec \lambda (\sigma )$.  Such a decomposition
does  exist because $\mu \prec \lambda (\sigma )$, and the cone of
  vectors that are orthogonal to $\mu $ and point towards
  $\lambda (\sigma )$ contains $v_{\sigma ,\mu ^{\perp}}$ and is
  generated by the vectors $v_{\nu \setminus \mu }$ for
  $\nu \in \Pi|_{V} (l+1)$ with
  $\mu \prec \nu \prec \lambda (\sigma )$.

Let $W$ be an open subset of $V$ such that $\Pi '|_{W}$ consists only
of the polyhedra having $\tau $ as a face.  We then define a
positive Minkowski weight $c \in M_{l+1}^{+}(\Pi |_{W})$ by
setting, for each $\nu \in \Pi|_{W}(l+1)$,
  \begin{equation}
    \label{eq:51}
    c(\nu )= \hspace{2mm} \sum_{\mathclap{\substack{\sigma\succ\tau\\  \lambda (\sigma
        )\succ \nu}}} a_{\sigma, \nu }\, c'(\sigma ),
  \end{equation}
  the sum being over the polyhedra $\sigma \in \Pi '|_{V}(k)$ with
  $\sigma\succ\tau$ and $ \lambda (\sigma )\succ \nu$.  Since
  $\Pi |_{W}(l)=\{\mu\}$, to see that the formula in \eqref{eq:51}
  defines a positive Minkowski weight we only need to check the
  balancing condition at this polyhedron:
  \begin{multline*}
    \sum_{\nu \succ \mu }  c(\nu )\, v_{\nu \setminus \mu }
    =\sum_{\nu \succ \mu } \sum_{\substack{\sigma\succ\tau\\  \lambda (\sigma )\succ \nu}} 
    a_{\sigma ,\nu } \, c'(\sigma ) \, v_{\nu \setminus \mu }\\
    =\sum_{\sigma\succ\tau} c'(\sigma )  
\bigg(\hspace{4mm}    \sum_{\mathclap{\mu \prec \nu
    \prec \lambda (\sigma ) }}
    a_{\sigma, \nu }\, v_{\nu \setminus \mu }\bigg)
=\sum_{\sigma\succ\tau}    c'(\sigma )\, v_{\sigma ,\mu ^{\perp}}=0,
  \end{multline*}
  where the sums are taken over the polyhedra $\nu\in \Pi|_{W}(l+1)$
  and $\sigma \in\Pi'|_{W}(k)$ subject to the conditions stated
  therein.  Moreover, the product $(f\cdot c')(\tau )$ coincides with
  $(f\cdot c)(\mu )$:
  \begin{align*}
    (f\cdot c')(\tau )&= -\sum_{\sigma\succ\tau}  c'(\sigma )\,  f_{\sigma }(v_{\sigma \setminus \tau })\\
                     &= -\sum_{\sigma\succ\tau} c'(\sigma )\, f_{\lambda (\sigma ) }(v_{\sigma ,\mu
    }+v_{\sigma ,\mu ^{\perp}})\\ &=
    -\sum_{\sigma\succ\tau} c'(\sigma )\,  f_{\mu }(v_{\sigma ,\mu}) -
    \sum_{\sigma\succ\tau}
\hspace{4mm}    \sum_{\mathclap{\mu \prec \nu
    \prec \lambda  (\sigma ) }} c'(\sigma )\, a_{\sigma, \nu }\, 
    f_{\lambda (\sigma ) }(v_{\nu\setminus \mu })\\
    &=- f_{\mu }(0)- \sum_{\mu \prec \nu   } c(\nu )\,  f_{\nu}(v_{\nu \setminus \mu
    })\\
    &=(f\cdot c)(\mu ) ,
  \end{align*}
  where the indexes $\nu$ and $\sigma$ denote polyhedra in
  $\Pi|_{W}(l+1)$ and in $\Pi'|_{W}(k)$, respectively: the first
  equality is the definition of the product, the second follows from
  the decomposition in \eqref{eq:41} and the fact that
  $ f_{\sigma}$ and $ f_{\lambda(\sigma)}$ coincide on
  $\iota_{X}(N_{\sigma})$, the third from the decomposition in
  \eqref{eq:43} and the fact that $ f_{\lambda(\sigma)}$ and
  $ f_{\mu}$ coincide on $\iota_{X}(N_{\mu})$, the fourth from
  the first equation in \eqref{eq:28}, the definition of the Minkowski
  weight $c$ and the fact that $ f_{\lambda(\sigma)}$ and
  $ f_{\nu}$ coincide on $\iota_{X}(N_{\nu})$, and the fifth
  from the definition of the product.

  Since $c$ is positive, by hypothesis $(f\cdot c)(\mu ) \ge0$.  Hence
  $(f\cdot c')(\tau )\ge 0$ and since $\tau \in \Pi'|_{V}(k-1)$ is
  arbitrary, we deduce that $f\cdot c'$ is positive and so
  \begin{displaymath}
f\cdot \gamma = f\cdot [c'] = [f\cdot c']   
  \end{displaymath}
  is also positive. Varying $V$, $k$ and $\gamma$, we conclude that
  $f$ is concave.
\end{proof}

The next three examples show that the proposed classes of concave
piecewise affine functions are different.  

 \begin{exmpl}\label{exm:4} 
  Let $X$ be the $1$-dimensional polyhedral space made of the three
  polyhedra $\sigma _{i}=\R_{\ge 0}\times \{i\}$, $i=1,2,3$, glued
  together by the points $(0,i)$. Let $H=\R$ with its standard
  Euclidean structure and let $\iota \colon X \to H$ be the
  quasi-embedding defined, for $(x,i) \in X$, by
  \begin{displaymath}
    \iota(x,i)=
    \begin{cases}
      -x&\text{ if }i=1,\\
      x&\text{ if }i=2,3.
    \end{cases}
  \end{displaymath}
  Consider the polyhedral complex $\Pi$ on $X$ made of the three rays
  $\sigma_{i}$ and the origin $\tau=(0,1)$ and, for
  $\gamma_{i}\in \R_{>0}$, $i=2,3$, consider the balancing condition
  $b\in M_{1}(\Pi)$ given by
  \begin{displaymath}
    b(\sigma_{1})=\gamma_{2}+\gamma_{3}, \quad      b(\sigma_{2})=\gamma_{2} \and
    b(\sigma_{3})=\gamma_{3}.
  \end{displaymath}

  A piecewise affine function $f\colon X \to \R$ defined on $\Pi$
  is of the form $f(x,i)=a_{i}\, x$ with $a_{i}\in
  \R$. Then 
  \begin{enumerate}
  \item \label{item:7} $f$ is strongly
  concave if and only if  $-a_{1}\ge a_{2}=a_{3}$,
  \item \label{item:13} $f$  is concave if and only if 
  $-a_{1}\ge a_{2}$ and $-a_{1}\ge a_{3}$,
\item \label{item:16} $f$ is weakly concave if and only if
  $-(\gamma_{2}+\gamma_{3}) \, a_{1}\ge \gamma_{2}\, a_{2}+
  \gamma_{3}\, a_{3}$.
\end{enumerate}
Hence these
classes of concave piecewise affine functions are different, and the
notion of being weakly concave depends on the choice of the balancing
condition.
  \end{exmpl}

  The quasi-embedding in the previous example is not globally
  injective. In the next examples we will see that even when the
  quasi-embedding is a injective, the three notions of concavity can
  differ.

\begin{exmpl}\label{exm:7}
  Let $X $ be the $1$-dimensional balanced polyhedral space in
  Example~\ref{exm:8}. Let $\Pi$ the polyhedral complex on $X$ made of
  the four rays $\sigma_{i} = \R_{\ge 0}\times \{i\}$, $i=1,\dots, 4$, and their origin
  $\tau=(0,1)$, and $U$ an open neighborhood of this point.

  The cone of positive Minkowski weights on $\Pi|_U$ of dimension 1 is generated by
  the weights $c_{1}$ and $c_{2}$ given by
  \begin{gather*}
    c_{1}(\sigma _{1})=1, \quad c_{1}(\sigma _{2})=1, \quad
    c_{1}(\sigma _{3})=0, \quad c_{1}(\sigma _{4})=0,\\
        c_{2}(\sigma _{1})=0,\quad c_{2}(\sigma _{2})=\sqrt{2},\quad
    c_{1}(\sigma _{3})=1, \quad c_{1}(\sigma _{4})=1.
  \end{gather*}
The piecewise affine function $f\colon X\to \R$ given by the formula \eqref{eq:22} is
defined on $\Pi$ and we have that 
  \begin{displaymath}
    (f\cdot c_{1})(\tau)=1 \and (f\cdot c_{2})(\tau)=0.
  \end{displaymath}
  By Proposition \ref{prop:13} it is concave although, as it was
  explained in Example \ref{exm:8}, it is not strongly concave.
\end{exmpl}

In the next example we construct a weakly concave piecewise affine
function that multiplied twice with the balancing condition gives a
negative weight. In particular, such a
piecewise affine function is not concave.

\begin{exmpl}
  \label{exm:9}
Let  $e_{i}$, $i=1,\dots,5$, be five linearly independent vectors in a
vector space $M$ and consider the polyhedral space $X$ in $M$ made of
the nine $2$-dimensional cones $\sigma_{i}$, $i=1,\dots, 9$,
respectively generated by the pairs
$\{ e_{1},e_{3}\}$,
  $\{ e_{1},e_{4}\}$, $\{ e_{1},e_{5}\}$,
  $\{ e_{2},e_{3}\}$, $\{ e_{2},e_{4}\}$,
  $\{ e_{2},e_{5}\}$, $\{ e_{3},e_{4}\}$,
  $\{ e_{3},e_{5}\}$ and $\{ e_{4},e_{5}\}$.
Let
  $\iota \colon X \to H=\R^{3}$ be the quasi-embedding given by
  \begin{alignat*}{2}
    \iota (e_{1})&=(0,0,1),& \quad \iota (e_{2})&=(0,0,-1),\\
    \iota (e_{3})&=(1,0,0),& \quad  \iota (e_{4})&=(0,1,0), \quad 
    \iota (e_{5})=(-1,-1,0).
  \end{alignat*}  
  Let $\Pi$ be the polyhedral complex on $X$ made of the origin of
  $M$, the five rays $\tau _{i} =\R_{\ge 0}\, e_{i}$, $i=1,\dots,5$,
  and the nine cones $\sigma_{i}$, $i=1,\dots, 9$. Let $b\in M_{2}(\Pi)$ be
  the balancing condition on $\Pi$ given by
  \begin{displaymath}
   b(\sigma_{i})=
   \begin{cases}
     \sqrt{2} &\text{ if } i=3, 6 ,\\
     1 &\text{ otherwise}.
   \end{cases}
  \end{displaymath}
  Hence $X$ is  balanced and its quasi-embedding is
  globally injective.

  Let $f\colon X \to \R$ be the 
  piecewise affine function defined on $\Pi$  given by  the values
  \begin{displaymath}
    f(e_{1})=1,\quad f(e_{2})=0,\quad f(e_{3})=-1,\quad f(e_{4})=0,\quad f(e_{5})=0.  
  \end{displaymath}
  We have that
  \begin{displaymath}
    (f\cdot b)(\tau _{1})=1,\ (f\cdot b)(\tau _{2})=1,\
    (f\cdot b)(\tau _{3})=0,\ (f\cdot b)(\tau _{4})=0,\
    (f\cdot b)(\tau _{5})=0.  
  \end{displaymath}
  Hence the $1$-dimensional Minkowski weight $f\cdot b$ is positive,
  and so $f$ is weakly concave. On the other hand
  $ (f\cdot (f\cdot b))(\{0\})=-1$, and therefore $f$ is not concave.
  \end{exmpl}

  We next show that weakly concave piecewise affine functions preserve
  the positivity of the Minkowski weights given by products of concave
  piecewise affine functions with the balancing condition.  By
  contrast, the product of several weakly concave piecewise affine
  functions with the balancing condition is not necessarily positive,
  as already shown in Example \ref{exm:9}.

  \begin{defn} \label{def:8}
    Suppose that the open subset $U \subset X$ is
    balanced with balancing condition $\beta _{U}$. Then a Minkowski
    cycle $\gamma$ on $U$ of dimension $k$ 
    is \emph{$\beta _{U}$-positive} if there are positive real numbers
    $\alpha_{j}$, 
    $j=1,\dots,s$ and concave piecewise affine functions
    $f_{1,j}, \dots, f_{n-k,j}$ on $U$, $j=1,\dots,
    s$,  such that
    \begin{displaymath}
      \gamma =\sum_{j=1}^{s}\alpha _{j} \,  f_{1,j} \cdots
      f_{n-k,j}  \cdot
      \beta_{U}.
    \end{displaymath}
    A Minkowski weight on $U$ is \emph{$\beta _{U}$-positive} if it
    represents a $\beta _{U}$-positive Minkowski cycle.
\end{defn}
 
\begin{prop}
  \label{prop:15}
  Suppose that $U \subset X$ is balanced, and let $f $ be a weakly
  concave piecewise affine function and $\gamma$ a
  $\beta _{U}$-positive Minkowski cycle on $U$. Then $f\cdot \gamma$
  is a $\beta _{U}$-positive Minkowski cycle on $U$.
  \end{prop}
  
  \begin{proof}
    It is enough prove the statement for Minkowski cycles of the form
 \begin{displaymath}
   \gamma =   f_{1} \cdots f_{n-k} \cdot 
   \beta_{U} \in Z_{k}(U)
 \end{displaymath}
 with $0 \le k\le n$ and $f_{i}\in \CPA(U)$, $i=1,\dots, n-k$.  By
 Proposition \ref{prop:4}, we have that
 $ f\cdot \gamma = f_{1} \cdots f_{n-k} \cdot f\cdot \beta_{U} $ and
 the definitions of concave and weakly concave piecewise affine
 functions imply that this product is positive, as stated. 
\end{proof}

    \begin{prop}
      \label{prop:7} 
      For a piecewise affine function on a convex open subset $U$ of
      $\R^{n}$ with the structure of balanced polyhedral complex of
      Example \ref{exm:2},
    each of the conditions of being weakly concave, concave,
    locally strongly concave and strongly concave is equivalent to
    being concave in the usual sense.
  \end{prop}

  \begin{proof}
    Let $f$ be a piecewise
    affine function on $U$. If $f$ is concave in the usual sense, then
    it is also locally strongly concave and, by
    Proposition~\ref{prop:11}, it also satisfies the other notions of
    concavity.

    For the converse, it is enough to show that the notion of being
    weakly concave implies that of being concave in the usual sense.
    Let $\Pi$ be a polyhedral complex on $\R^{n}$ with the induced
    balancing condition $b\in M_{n}^{+}(\Pi|_{U})$, and let
    $f\in \WCPA_{\Pi}(U)$.

    For each $\tau\in \Pi|_{U}(n-1)$ denote by
    $\sigma,\sigma'\in \Pi|_{U}(n)$ the two polyhedra containing
    $\tau$ as a facet. Since $f$ is weakly concave,
    \begin{equation}\label{eq:44}
      (f\cdot b)(\tau) = - f_{\sigma}(v_{\sigma\setminus \tau})
      - f_{\sigma'}(v_{\sigma'\setminus \tau}) \ge 0.
    \end{equation}
    The linear functions $ f_{\sigma}$ and $ f_{\sigma'}$
    coincide on $\iota_{X,\tau}(N_{\tau})$ and
    $v_{\sigma\setminus \tau} = -v_{\sigma'\setminus \tau} $, and so
    the inequality in \eqref{eq:44} implies that the restriction of
    $f$ to any convex subset of the union $\sigma\cup\sigma'$ is
    concave in the usual sense.

    Now let $x,x' \in U$ be two different points such that the segment
    $\overline{x x'}$ does not intersect the $(n-2)$-dimensional
    skeleton of $\Pi|_{U}$. Then every point
    $y \in \overline{x x'}$ has a convex neighborhood that is
    contained in a subset of the form $\sigma\cup\sigma'$ as
    above, and the  restriction of $f$ to such a  neighborhood is
    concave in the usual sense. Since this condition is local, we
    deduce that $f$ is concave on the segment $\overline{xx'}$.

    Hence, $f$ is concave on each segment $\overline{x x'}$ for the
    pairs $(x,x')$ in a dense open subset of $U\times U$.  Since $f$
    is continuous (Remark \ref{rem:1}), this property extends to every
    segment in $U$ and so $f$ is concave therein, as stated.
  \end{proof}

  \begin{cor}\label{cor:1} Let $X$ be a Euclidean  polyhedral
    space of pure dimension, $U$ a balanced open subset of $X$, and
    $f$ a weakly concave piecewise affine function on $U $. Let $\Pi $
    be a polyhedral complex on $X$ and $\sigma \in \Pi $ a polyhedron
    contained in $U$. Then the restriction $f|_{\sigma}$ is concave in
    the usual sense.
  \end{cor}
  
\begin{proof}
  We can assume without loss of generality that $\sigma $ is of
  maximal dimension. The balancing condition on $U$ induces a
  balancing condition on the interior of $\sigma $ that is a positive
  multiple of the standard balancing condition of a polyhedron of
  $\R^{n}$.  By Proposition \ref{prop:7} the restriction of $f$ to the
  interior of $\sigma $ is concave in the usual sense.  Since $f$ is
  continuous, it is also concave on the whole of~$\sigma $.
\end{proof}

  \begin{defn}
    \label{def:projective}
    Let $\Pi$ be a polyhedral complex on the Euclidean polyhedral
    space $X$. A piecewise affine function $f$ on the open subset $U$
    is \emph{strictly concave on~$\Pi|_{U}$} if it is defined on $\Pi$
    and for every integer $k$, every $\tau\in \Pi|_{U}(k)$, every open subset
    $V\subset U$ with $V\cap \tau\ne \emptyset$ and every
    $c\in M^{+}_{k+1}(\Pi |_{V})$ with $c(\sigma)>0$ for some
    $\sigma\in \Pi|_{V}(k+1)$ with $\sigma\succ\tau$, we have that
    \begin{equation}
      \label{eq:48}
    (f\cdot c)(\tau)>0  .
    \end{equation}
    The polyhedral complex $\Pi$ is \emph{regular} on $U$ if there is
    a piecewise affine function $f$ on $U$ that is strictly concave on
    $\Pi|_{U}$.
  \end{defn}
  
  \begin{rem}
    The notion of strictly concave introduced in definition
    \ref{def:projective} is a direct generalization of the notion of
    strictly convex function on a fan used in toric geometry, see for
    instance \cite[page 68]{Fulton:itv}. But take into account that
    what is called convex in \emph{loc. cit.}, is called concave in
    the context of convex analysis.
  \end{rem}

For a convex polyhedron $\Delta$ in a Euclidean vector space, a
concave piecewise affine function $f $ on it induces a polyhedral
complex on $\Delta$, denoted by
\begin{equation}
  \label{eq:49}
  \Pi(f),
\end{equation}
given by its affinity loci or equivalently, by the projection in the
convex polyhedron of the faces of the upper envelope of the hypograph
of $f$ . In more explicit terms, a subset $\sigma \subset \Delta$ is
an element of $\Pi(f)$ if and only if there is an affine function
$\ell \colon \Delta\to \R$ such that
\begin{equation}
  \label{eq:24}
    f(x)=\ell(x) \text{ for } x\in \sigma \and  f(x)<\ell(x) \text{ for } x\in
\Delta \setminus \sigma.  
\end{equation}
Indeed, this is the minimal polyhedral complex on $\Delta$ where $f$ is
defined.

Classically, a polyhedral complex on a convex polyhedron is
called regular if it is of the form $\Pi(f)$ for a concave piecewise
affine function on it, see for instance \cite[Definition
2.2.10]{DeLoeraRambauSantos:tsaa}.  The next result shows that
Definition \ref{def:projective} is a generalization of this notion to 
arbitrary Euclidean polyhedral complexes.

\begin{prop}
  \label{prop:18}
  Let $\Delta$ be a full dimensional polyhedron of $\R^{n}$ with the
  induced structure
  of Euclidean polyhedral space and set $U=\relint(\Delta)$.  Let
  $\Pi$ be a polyhedral complex on $\Delta$ and $f\in \PA_\Pi(\Delta
  )$. Then $f$ is strictly concave on $\Pi|_{U} $ if and only if it is
  concave in the usual sense and $\Pi =\Pi (f)$.

  In particular, $\Pi$ is regular on $U$ in the sense of Definition
  \ref{def:projective} if and only if there is a concave piecewise
  affine function $f$ on $\Delta$ such that $\Pi=\Pi(f)$.
\end{prop}
\begin{proof}
  First assume that $f$ is strictly concave on $\Pi |_{U}$, and let
  $b\in M_{n}(\Pi|_{U})$ be the balancing condition on $\Pi|_{U}$
  induced from the standard one of $\R^{n}$, given in
  Example~\ref{exm:2}. Since $b(\sigma)>0$ for all
  $\sigma \in \Pi|_{U}(n)$, the condition \eqref{eq:48} implies that
  $f$ is weakly concave on $U$. By Proposition \ref{prop:7}, it is
  concave on $U$ in the usual sense and by continuity, it is also
  concave on the whole of $\Delta$. Since $f$ is defined on $\Pi$, we
  have that $\Pi(f)\le \Pi$. Moreover, the condition that
  $(f\cdot b)(\tau) >0$ for all $\tau\in \Pi|_{U}(n-1)$ implies that
  $\Pi(f)= \Pi$, proving the first implication.

  Conversely, suppose that $f$ is concave on $\Delta$ and
  $\Pi=\Pi(f)$.  For each integer $0< k\le n$ and each $\tau\in
  \Pi|_{U}(k)$ choose an affine 
  function $\ell$ as in~\eqref{eq:24}.  Choose then a point
  $x\in \relint(\tau )$ and $\varepsilon >0$ such that, for all
  $\sigma \in \Pi (f)({k+1})$ with $\sigma \succ\tau $ the
  condition $x+\varepsilon \, v_{\sigma \setminus \tau }\in \sigma $
  holds. Then for $c\in M^{+}_{k+1}(\Pi |_{V})$ with $c(\sigma)>0$ for
  some $\sigma\in \Pi|_{V}(k+1)$ with $\sigma\succ\tau$, we have that
  \begin{displaymath}
        (f\cdot c)(\tau )=\frac{1}{\varepsilon }\sum_{\sigma \succ \tau
    }c(\sigma )(f(x)-f(x+\varepsilon \, v_{\sigma
      \setminus \tau }))
    >\frac{1}{\varepsilon }\sum_{\sigma \succ \tau }c(\sigma
    )(\ell(x)-\ell(x+\varepsilon \, v_{\sigma \setminus \tau })) = 0,
  \end{displaymath}
  thanks to the formula in \eqref{eq:2}, the inequality
  in~\eqref{eq:24}, and the fact that $\ell$ is affine. This shows
  that $f$ is strictly concave on $\Pi$ and completes the proof.
\end{proof}

\begin{rem}
  The proof of Proposition \ref{prop:18} shows that in order to check that a
  piecewise affine function $f$ defined on a polyhedral complex $\Pi$
  on an $n$-dimensional polyhedron $\Delta $ is strictly concave, it
  suffices to show that $(f\cdot b)(\tau )> 0$ for every
  $\tau \in \Pi (n-1)$. This is not the case when $\Delta$ is replaced
  by an arbitrary polyhedral space.
\end{rem}

Let $\Pi $ be a polyhedral complex on $X$, and $f$ a piecewise affine
function on this Euclidean polyhedral space whose restriction to each
polyhedron of $\Pi$ is concave. Then we consider the set of polyhedra
in $X$ defined as
  \begin{equation}
    \label{eq:46}
    S(\Pi, f)=\bigcup_{\sigma\in \Pi} \Pi(f|_{\sigma}),
  \end{equation}
  where for each $\sigma\in \Pi$ we denote by $\Pi(f|_{\sigma})$ the
  polyhedral complex on $\sigma$ given by the affinity loci of the
  restriction of $f$ to this polyhedron as in \eqref{eq:49}.

  \begin{prop}\label{prop:16}
    The set $S(\Pi, f)$ is the minimal subdivision of $\Pi $ where the
    piecewise affine function $f$ is defined.
\end{prop}

\begin{proof}
  For each $\tau, \sigma\in \Pi$ with $\tau\prec \sigma$, the
  restriction to $\tau $ of the polyhedral complex $\Pi(f|_{\sigma})$
  agrees with $\Pi(f|_{\sigma})$. This implies that $S(\Pi,f)$ is a
  polyhedral complex on $X$, and Proposition \ref{prop:18} implies
  that it is the minimal subdivision of $\Pi $ where $f$ is defined.
\end{proof}

We next show that any polyhedral complex satisfying a finiteness
condition admits a regular subdivision.

\begin{thm}\label{prop:projective} Let
  $\Pi $ be a polyhedral complex on $X$ such that $\Pi|_{U}$ is
  finite.  Then there is a strongly concave piecewise affine function
  $f$ on $X$ such that $S(\Pi, f)|_{U}$ is finite and $f$ is strictly
  concave on $S(\Pi ,f)|_{U}$.  In particular, $S(\Pi, f)$ is a
  subdivision of $\Pi$ that is regular on $U$.
\end{thm}
\begin{proof}
  For each $\sigma\in \Pi|_{U}$ choose linear functions
  $\ell_{\sigma,j}$ on $N_{X}$, $j=1,\dots, m_{\sigma}$, defining the
  polyhedron $\iota_{X}(\sigma)$ as the subset of points
  $x\in H_{X}$ such that $ \ell_{\sigma,j}(x)\ge 0$ for all
  $j$.  Consider the concave piecewise affine function
  $f'_{\sigma}\colon H_{X}\to \R$ defined, for
  $x \in H_{X}$, as
\begin{equation}\label{eq:14}
    f'_{\sigma }(x)=\min(0,\ell_{\sigma,1}(x), \dots,
    \ell_{\sigma,m_{\sigma}}(x) ).
  \end{equation}
Set
  \begin{displaymath}
    f'= \hspace{2mm}\sum_{\mathclap{\sigma \in {\Pi|_U} }}f'_{\sigma } \and  f=f'\circ \iota_{X}.
  \end{displaymath}
  Since $\Pi|_{U}$ is finite, $f'$ is well-defined, concave and
  piecewise linear. Since $f$ is the pull-back to $X$ of the concave
  function $f'$ on $H_{X}$, it is a strongly concave piecewise affine
  function on $X$. By Proposition \ref{prop:16}, we can then consider
  the associated subdivision $S(\Pi,f)$ of $\Pi$. By construction, the
  restriction of $S(\Pi,f)$ to the open subset $U$ is also finite.

  We need to show that $f$ is strictly concave on $S(\Pi,f)$.
  So let $\tau \in S(\Pi,f)|_U(k-1)$ with $k\ge 1$. Let
  $c\in M_{k}^{+}(S(\Pi,f)|_{U})$ with $c(\sigma)>0$ for some
  $\sigma\in S(\Pi,f)|_{U}(k)$ with $\sigma\succ\tau$.  Let $\rho$ be the
  minimal polyhedron of $\Pi|_{U}$ containing $\tau$.

  If $\dim(\rho ) > \dim(\tau)$ then necessarily
  $ \tau \cap \relint(\rho)\ne \emptyset$ since otherwise $\tau $
  would be contained in a face of $\rho $, contradicting the
  minimality of $\rho $. In this case Proposition~\ref{prop:18}
  implies that $(f\cdot c)(\tau)>0$. Else when
  $\dim(\rho )=\dim(\tau )$, we number as
  $\sigma _{1},\dots,\sigma _{r}$ the polyhedra in $ S(\Pi ,f)|_V(k)$
  with $\tau \prec \sigma _{i}$ and $c(\sigma _{i})>0$. We have that
  $r>0$ by the hypothesis on $c$. Choose a point $x$ in the relative
  interior of $\tau $ and $\varepsilon >0$ such that, for
  $i=1,\dots,r$,
\begin{equation}
  \label{eq:56}
  x_{i}\coloneqq 
  x+\varepsilon \, v_{\sigma _{i}\setminus
    \tau }\in \relint(\sigma _{i}).
\end{equation}
Write $\alpha _{i}={c(\sigma _{i})}/{\sum_{j=1}^{r}c(\sigma
  _{j})}$. Since $c$ is a Minkowski weight, we have that 
$ \iota_{X}(x)=\sum \alpha _{i}\, \iota_{X}(x_{i})$.  Moreover,
\begin{displaymath}
  (f\cdot c)(\tau )=\frac{\sum_{j} c(\sigma _{j})}{\varepsilon
  }\left(f'(\iota_{X}(x))-\sum_{i=1}\alpha _{i}\, f'(\iota_{X} (x_{i}))\right).
\end{displaymath}
Since $\tau \subset \rho $, $\dim(\tau )=\dim(\rho )$ and $x_{i}\in
\relint(\sigma _{i})$ where $\dim(\sigma _{i})>\dim(\tau )$ and
$\sigma _{i}\succ \tau $, we deduce that $x_{i}\not\in \rho $. Let
$f'_{\rho }$ be the function defined as \eqref{eq:14} for $\rho
$. Then $ f'_{\rho }(\iota_{X}(x))=0$ and, for each $i$, we also
have that $f'_{\rho }(\iota_{X}(x_{i})) <0$ because $x_{i}\notin
\rho $. This implies that
\begin{equation*}
  f'_{\rho }(\iota(x))-\sum_{i=1}\alpha _{i}f'_{\rho }(\iota (x_{i}))>0.
\end{equation*}
Since $f'$ is obtained from $f'_{\rho }$ by adding to it a concave
piecewise affine function, we deduce that
\begin{displaymath}
  f'(\iota(x))-\sum_{i=1}\alpha _{i}f'(\iota (x_{i}))>0.
\end{displaymath}
Hence $f$ is strictly concave on $S(\Pi,f)|_U$ and $S(\Pi,f)$ is a
subdivision of $\Pi$ that is regular on $U$, as stated.
\end{proof}
   
\begin{thm}\label{thm:1} 
Let $\Pi$ be a polyhedral complex on $X$ such that $\Pi|_{U}$ is
  finite. Then there is a subdivision $\Pi'$ of $\Pi$ such that for
  every piecewise affine function $g\in \PA_{\Pi'}(U)$ there are concave
  piecewise affine functions $f_{1},f_{2}\in \CPA_{\Pi'}(U)$
  with
    \begin{displaymath}
      g=f_{1}-f_{2}.
    \end{displaymath}
  \end{thm}
\begin{proof}
  By Theorem \ref{prop:projective}, there is a subdivision $\Pi'$ of
  $\Pi$ and a strictly concave function $f$ on $\Pi'|_{U} $.  The set
  $ \PA_{\Pi'}(U)$ is a finite dimensional vector space, and
  $\CPA_{\Pi'}(U)$ is a cone inside it. For each $k\in \Z_{\ge 0}$ and
  $\tau \in \Pi(k-1)$ choose an open subset $V\subset U$ such that
  $\tau$ is the only polyhedron of $\Pi(k-1)$ intersecting $V$. Each
  Minkowski weight $c\in M_{k}(\Pi|_{V})$ gives a functional
  $ \ell_{k,\tau,V,c}\colon \PA_{\Pi'}(U) \rightarrow \R $ defined,
  for $g\in \PA_{\Pi'}(U)$, by
   \begin{equation}
    \label{eq:45}
 \ell_{k,\tau,V,c}(g)= (g\cdot c)(\tau).  
  \end{equation}
  Since $f$ is strictly concave on~$\Pi|_{V}$, we have that
  $\ell_{k,\tau,V,c}(f)>0$ for all $k$, $\tau$, $V$ and $c$ as above.
  Hence $f$ lies in the interior of the cone $\CPA_{\Pi'}(U)$, and so
  the interior of this cone is nonempty. Thus
  \begin{displaymath}
\PA_{\Pi'}(U)=\CPA_{\Pi'}(U) - \CPA_{\Pi'}(U),
  \end{displaymath}
which gives the result.
\end{proof}

\begin{rem}\label{rem:fairly-concave}  
  The class of concave piecewise affine functions satisfies two
  fundamental properties, as it follows from its definition and
  Theorem \ref{thm:1}:
\begin{enumerate}
  \item \label{item:1}  they preserve the positivity of Minkowski
    weights,
  \item \label{item:2} piecewise affine functions can be written as a
    difference of concave piecewise affine functions.
  \end{enumerate}
  By contrast, strongly concave piecewise affine functions do not
  satisfy the condition~\eqref{item:2} and weakly concave functions do
  not satisfy the condition \eqref{item:1}, as it can be verified in
  the situation of Example \ref{exm:4}
\end{rem}
  
\section{Concave functions on polyhedral spaces}
\label{sec:conc-funct-polyh}

In this section we given different notions of concavity for functions on polyhedral spaces,
extending the previous notions  for piecewise affine
functions. We denote by $X$ a quasi-embedded polyhedral space with
quasi-embedding $\iota\colon X\to N$, and $U$ an open subset of it.

First we first introduce notions of convex combinations of points.

\begin{defn}\label{def:5}
  A \emph{convex combination} in $U$ is a triple
  \begin{equation*}
   (x,(x_{i})_{i\in I}, (\nu _{i})_{i\in I})  
 \end{equation*}
 with $x\in U$, and where 
 $ (x_{i})_{i\in I}$ and $ (\nu
 _{i})_{i\in I}$ respectively denote finite collections of points of $U$ and  of
 nonnegative real numbers satisfying 
    \begin{displaymath}
   \sum_{i\in I}\nu _{i}=1 \and    \sum_{i\in I}\nu _{i} \,
   \iota(x_{i}) =\iota (x).
 \end{displaymath}
 
 A convex combination in $U$ is \emph{polyhedral} if there is a
 polyhedral complex $\Pi$ on $X$ and  polyhedra $\tau\in \Pi$ with
 $x\in \tau $ and $\sigma _{i}\in \Pi$, $i\in I$, such that
 $\sigma_{i}\succ \tau$, $\sigma _{i}\subset U$ and
 $x_{i}\in \sigma _{i}$.

   If $X$ is Euclidean and of pure dimension, and $U$ is
 balanced, a convex combination in $U$ is \emph{balanced} if it is
 polyhedral and, with notation as above and setting  $k=\dim(\tau)+1$,
satisfies the following conditions:
 \begin{enumerate}
 \item \label{item:5} the polyhedra $\sigma _{i}$, $i\in I$, are the
   different polyhedra in $ \Pi(k)$ having $\tau$ as a face,
 \item \label{item:17} there is a $\beta _{U}$-positive Minkowski
   weight $c\in M_{k}(\Pi|_{U} )$ (Definition ~\ref{def:8}) such that
$\nu _{i }\, d(x_{i },{\tau} ) =  c(\sigma _{i})$,
   $i\in I$, where $d(x_{i },{\tau} )$ denotes the distance between
   the point $\iota(x_{i})$ and the affine subspace
   $\iota_{\tau}(H_{\tau })$ of $ N$.
\end{enumerate}
\end{defn}

We next give examples  illustrating the distinction between the
different notions of convex combination.

\begin{exmpl}\label{exm:3}
  We place ourselves in the setting of Example \ref{exm:4}.  For
  $p\in \R_{>0}$  set
    \begin{displaymath}
      x=(p,2), \quad x_{1}=(p,3) \and 
      \nu _{1}=1.   
    \end{displaymath}
    Then $(x,(x_{1}),(\nu _{1}))$ is a convex combination since
    $\iota(x)=\iota(x_{1})=p$ and so we have that
    $ \sum_{i=1}^{1}\nu_{i}\, \iota(x_{i})=\iota(x)$.  However, this
    convex combination is not polyhedral: there is no polyhedral
    complex on $X$ such that $x_{1}$ lies in a polyhedron with a face
    containing the point $x$.
    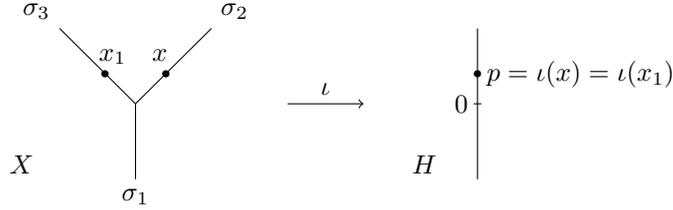
\begin{figure}[h]
\begin{center}
\begin{tikzpicture}[scale=1]
    \draw (0,0) -- (1,1);
    \draw (0,0)--(-1,1);
    \draw (0,0)--(0,-1); 
    \draw (0.4, 0.4) node{\tiny{$\bullet$}};
     \draw (0.4, 0.4) node[above]{$x_{\phantom{1}}$};
     \draw (-0.4,0.4) node{\tiny{$\bullet$}};
    \draw (-0.3,0.4) node[above]{$x_1$};
    \draw (-1.5,-0.8) node{$X$};
    \draw (1,1) node[above right]{$\sigma_2$};
     \draw (-1,1) node[above left]{$\sigma_3$};
     \draw (0,-1) node[below]{$\sigma_1$};
     \draw[->] (2,0) to node[above]{$\iota$} (3,0); 
    
    \draw (4.5,-1) -- (4.5,1);
     \draw (4.5,0) node[left]{$0$};
     \draw (3.8,-0.8) node{$H$};
     \draw (4.5,0.4) node{\tiny{$\bullet$}};
     \draw (4.5,0.4) node[right]{{$p = \iota(x) = \iota(x_1)$}};
     \draw[thin] (4.45,0) -- (4.55,0);
\end{tikzpicture}
\end{center}
\caption{A convex combination that is not polyhedral}\label{fig:conv-comb1}
\end{figure}

Set also
  \begin{displaymath}
x=(0,1), \quad x_{1}=(p,1), \quad x_{2}=(p,2) \and \nu _{i}=\frac{1}{2}, \quad i=1,2. 
  \end{displaymath}
  This is a polyhedral convex combination for the polyhedral complex
  $\Pi$ in Example~\ref{exm:4}: we have that $x\in \tau$ and
  $x_{i}\in \sigma _{i}$, $i=1,2$, and so
  $ \sum_{i=1}^{2}\nu _{i}\, \iota(x_{i})=0=\iota(x)$. This polyhedral
  convex combination is not balanced: setting $\nu_3 = 0$, the vectors
    \begin{displaymath}
      (\nu _{i}\, d(x_i,x))_{i=1,2,3}=\Big( \frac{p}{2}, \frac{p}{2},0 \Big) \and
      (b(\sigma _{i}))_{i=1,2,3}=(2,1,1)
    \end{displaymath}
    are not proportional.
     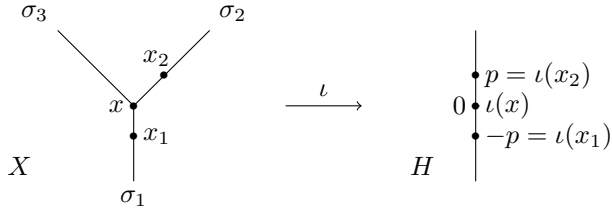
\begin{figure}[h]
\begin{center}
\begin{tikzpicture}[scale=1]
    \draw (0,0) -- (1,1);
    \draw (0,0)--(-1,1);
    \draw (0,0)--(0,-1); 
    \draw (0.4, 0.4) node{\tiny{$\bullet$}};
     \draw (0.3, 0.4) node[above]{$x_2$};
     \draw (0,-0.4) node{\tiny{$\bullet$}};
    \draw (0,-0.4) node[right]{$x_1$};
    \draw (0,0) node{\tiny{$\bullet$}};
    \draw (0,0) node[left]{$x$};
   \draw (-1.5,-0.8) node{$X$};
    \draw (1,1) node[above right]{$\sigma_2$};
     \draw (-1,1) node[above left]{$\sigma_3$};
     \draw (0,-1) node[below]{$\sigma_1$};
     \draw[->] (2,0) to node[above]{$\iota$} (3,0); 
     
      \draw (4.5,-1) -- (4.5,1);
       \draw (4.5,0) node[left]{$0$};
       \draw[thin] (4.45,0) -- (4.55,0);
       \draw (3.8,-0.8) node{$H$};
       
     \draw (4.5,0) node{\tiny{$\bullet$}};
     \draw (4.5,0) node[right]{{$\iota(x)$}};

      \draw (4.5,0.4) node{\tiny{$\bullet$}};
     \draw (4.5,0.4) node[right]{{$p = \iota(x_2)$}};
     \draw (4.5,-0.4) node{\tiny{$\bullet$}};
     \draw (4.5,-0.4) node[right]{{$-p = \iota(x_1)$}};
     
\end{tikzpicture}
\end{center}
\caption{A convex combination that is polyhedral but not balanced}\label{fig:conv-comb2}
\end{figure}

Finally, the points and positive real numbers
\begin{displaymath}
x=(0,1), \quad x_{1}=(2\, p,1), \quad x_{2}=(p,2), \quad x_{3}=(p,3) \and
\nu _{i}=\frac{1}{3}, \ i=1,2,3,
\end{displaymath}
do form a balanced  convex combination in $X$.
     \begin{figure}[h]
\begin{center}
\begin{tikzpicture}[scale=1]
     \draw (0,0) -- (1,1);
    \draw (0,0)--(-1,1);
    \draw (0,0)--(0,-1); 
    \draw (0.4, 0.4) node{\tiny{$\bullet$}};
     \draw (0.3, 0.4) node[above]{$x_2$};
      \draw (-0.4, 0.4) node{\tiny{$\bullet$}};
     \draw (-0.3, 0.4) node[above]{$x_3$};
     \draw (0,-0.8) node{\tiny{$\bullet$}};
    \draw (0,-0.8) node[right]{$x_1$};
    \draw (0,0) node{\tiny{$\bullet$}};
    \draw (0,0) node[left]{$x$};
   \draw (-1.5,-0.8) node{$X$};
    \draw (1,1) node[above right]{$\sigma_2$};
     \draw (-1,1) node[above left]{$\sigma_3$};
     \draw (0,-1) node[below]{$\sigma_1$};
     \draw[->] (2,0) to node[above]{$\iota$} (3,0); 
     
      \draw (4.5,-1) -- (4.5,1);
       \draw (4.5,0) node[left]{$0$};
       \draw[thin] (4.45,0) -- (4.55,0);
       \draw (3.8,-0.8) node{$H$};
       
     \draw (4.5,0) node{\tiny{$\bullet$}};
     \draw (4.5,0) node[right]{{$\iota(x)$}};
     
      \draw (4.5,0.4) node{\tiny{$\bullet$}};
     \draw (4.5,0.4) node[right]{{$p = \iota(x_2) = \iota(x_3)$}};
     \draw (4.5,-0.8) node{\tiny{$\bullet$}};
     \draw (4.5,-0.8) node[right]{{$-2p = \iota(x_1)$}};
     
\end{tikzpicture}
\end{center}
\caption{A balanced convex combination}\label{fig:conv-comb3}
\end{figure}
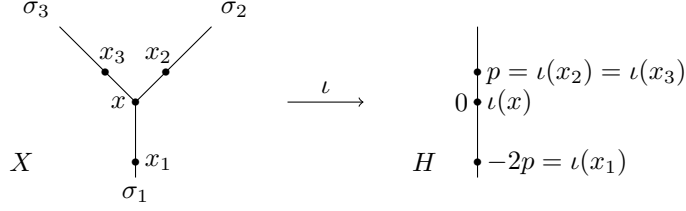
\end{exmpl}

The basic example of balanced convex combination arises when the
$\beta _{U}$-positive Minkowski cycle in Definition \ref{def:5} is the balancing
condition itself.

  \begin{exmpl}\label{exm:10}
    Suppose that $X$ is Euclidean, of pure dimension $n$ and
    balanced, and let $\Pi $ be a polyhedral complex on $X$ with
    balancing condition $b$.  For $\tau \in \Pi (n-1)$ let
    $\sigma _{i}$, $i=1,\dots, r$, be the different polyhedra in
    $\Pi (n)$ having $\tau$ as a facet and set
  \begin{displaymath}
    \nu _{i}=\frac{b(\sigma _{i})}{\sum_{i=1}^{r}b(\sigma _{i})},
    \quad i=1,\dots, r.
  \end{displaymath}
  Let $x\in \relint(\tau )$ and
  $(x,(y_{i})_{1\le i\le r}, (\nu_{i})_{1\le i\le r})$ a convex
  combination contained in $\relint(\tau)$. Let $\varepsilon >0 $ so
  that the point 
  \begin{displaymath}
    x_{i}\coloneqq y_{i}+\varepsilon \, v_{\sigma _{i}\setminus \tau } 
  \end{displaymath}
  lies in $\sigma _{i}$ for all $i$. Then
  $(x,(x_{i})_{1\le i\le r}, (\nu_{i})_{1\le i\le r})$ is a balanced
  convex combination.
\end{exmpl}

\begin{rem}
  \label{rem:4}
  The construction in Example \ref{exm:10} provides balanced convex
  combinations centered at any point in the relative interior of a
  polyhedron of dimension $n-1$. Considering arbitrary
  $\beta _{U}$-positive Minkowski weights as in Definition \ref{def:5}
  allows to produce balanced convex combinations centered at any point
  of $U$.
\end{rem}

\begin{defn}\label{def:1} 
  Let $f\colon U\to \R$ be a function. Then $f$ is \emph{strongly
    concave} if for every convex combination in $U$, the inequality
    \begin{equation}\label{eq:5}
    f(x)\ge \sum_{i\in I}\nu _{i} \, f(x_{i})
  \end{equation}
  holds. It is \emph{locally strongly concave} if there is an open
  covering $U=\bigcup_{i} U_{i}$ such that $f|_{U_{i}}$ is strongly
  concave for every $i$.

  The function $f$ is \emph{concave} if the inequality \eqref{eq:5} is
  satisfied for all polyhedral convex combinations in $U$.

  If $X$ is a Euclidean polyhedral space of pure dimension and $U$ is
  balanced, then $f$ is \emph{weakly concave} when this inequality is
  satisfied for all balanced convex combinations in $U$.
\end{defn}

The next result shows that Definition \ref{def:1} generalizes the
notions of concavity for piecewise affine functions from
Definitions~\ref{def:6} and \ref{def:3}.

\begin{prop}\label{prop:17} Let  $f$ be a piecewise affine function on
  $U$. Then $f$ is strongly concave in the sense of Definition
  \ref{def:6} (respectively concave, weakly concave in the
  sense of Definition \ref{def:3}) if and only if it is strongly
  concave (respectively concave, weakly concave) in the sense
  of Definition \ref{def:1}.
\end{prop}
\begin{proof}
  We prove separately each implication.
  
\smallskip
\noindent\emph{Strong concavity in the sense of Definition \ref{def:6}
    implies strong concavity in the sense of Definition \ref{def:1}:}

  Let $f$ be strongly concave in the sense of Definition
  \ref{def:6}. Then there is a function
  \begin{displaymath}
f' \colon N\longrightarrow \R\cup\{\pm\infty\}     
  \end{displaymath}
  that is concave in the sense of Definition \ref{def:7} and such that
  $f=f'\circ \iota|_{U}$.  Since $f$ takes only finite values, then
  $f'$ also takes only finite values on $\iota(U)$.
Given a convex combination
  $(x,(x_i)_{i \in I}, (\nu_i)_{i \in I})$ in $U$, the concavity of
  $f'$ and this finiteness condition imply that
\begin{displaymath}
f(x) = (f' \circ \iota)(x) = f' \Big(\sum_{i \in I}
  \nu_i \, 
  \iota (x_i)\Big) \geq \sum_{i \in I} \nu_i 
f'(\iota(x_i)) = \sum_{i \in I}\nu_i \, f(x_i),
\end{displaymath}
and so $f$ is strongly concave in the sense of Definition \ref{def:1}.
  
\smallskip
\noindent
\emph{Strong concavity in the sense of Definition \ref{def:1} implies
  strong concavity in the sense of Definition \ref{def:6}:}

Let now
$f$ be strongly concave in the sense of Definition \ref{def:1}.  Let
$g\colon \iota(U) \to \R $ be the function defined, for
$y\in \iota(U) $, by
\begin{displaymath}
  g(y)= f(x)  
\end{displaymath}
for any $x\in U$ with $\iota(x)=y$.  It is well-defined, because for
$x_1, x_2\in U$ such that $\iota(x_1) = \iota(x_2)$ we have that
$(x_1, (x_2), (1))$ is a convex combination in $U$ and so
$f(x_1) \geq f(x_2)$. Similarly $f(x_2) \geq f(x_1)$, and so
$f(x_1) = f(x_2)$.

Consider the function $f' \colon H\rightarrow \R\cup\{\pm\infty\} $
defined, for $y\in H$, by
\begin{equation}
  \label{eq:47}
f'(y)= \sup  \sum_{i\in I} \nu_i \, g(y_i),  
\end{equation}
the supremum being over the collections of points $ (y_i)_{i\in I}$ in
$\iota(U)$ and of nonnegative real numbers $(\nu_i)_{i \in I}$ with
$ \sum_{i\in I} \nu_{i}=1 $ and $ \sum_{i\in I} \nu_{i}\, y_{i}=y$. In
particular, if $y$ does not belong to the convex hull of
$\iota _{X}(U)$ then $f'(y)=-\infty$.

By construction, the hypograph of $f'$ coincides with the convex hull
of the subset $\{(y,z) \mid y\in \iota(U), z\le g(y)\}$ of
$H\times \R$. In particular, this hypograph is a convex set and so
$f'$ is concave in the the sense of Definition \ref{def:7}.

We claim that $f = f' \circ \iota|_{U}$. To prove this, let $x\in U$
and set $y = \iota(x)$. On the one hand, let $ (y_i)_{i\in I}$ and
$(\nu_i)_{i \in I}$ as in \eqref{eq:47} and write $y_i = \iota(x_i)$
with $x_i \in U$,
$i\in I$. Then $(x, (x_i)_{i\in I}, (\nu_i)_{i\in I})$ is a convex
combination in $U$ and so
\begin{displaymath}
  f(x) \geq \sum_{i\in I} \nu_i \, 
  f(x_i)= \sum_{i\in I}\nu_i\, 
g(y_i),
\end{displaymath}
which implies that $f'(y)\le f(x)$. 
On the other hand, 
choosing the collections
$ (y)$ and $ (1)$ in the definition of $f'$ in \eqref{eq:47}, we
deduce that
\begin{displaymath}
  f' (y)\ge f(x)
\end{displaymath}
and so $f'(y) = f(x)$. Hence $f = f' \circ \iota|_{U}$ and so $f$ is
strongly concave in the sense of Definition \ref{def:6}.

\smallskip
\noindent
\emph{Concavity in the sense of Definition \ref{def:3}
  implies concavity in the sense of Definition~\ref{def:1}:}

Now suppose that $f$ is concave in the sense of Definition \ref{def:3}
and consider a polyhedral convex combination
$(x, (x_{i})_{i\in I}, (\nu _{i\in I}))$ in $U$. By definition, there
is a polyhedral complex $\Pi $ on $X$ and polyhedra $\tau\in \Pi$ with
$x\in \tau $ and $\sigma _{i}\in \Pi$, $i\in I$, such that
$\sigma_{i}\succ \tau$, $\sigma _{i}\subset U$ and
$x_{i}\in \sigma _{i}$.

Choose a subdivision $\Pi'$ of $\Pi$ where $f $ is defined and such
that $x$ is a vertex of $\Pi'|_{U}$ and, for each $i\in I$, there is a
segment $\rho _{i}$ in $\Pi'|_{U}$ having $x$ as a vertex and
pointing in the direction of $x_{i}$ from $x$. Choose also
$\varepsilon >0$ such that the point
\begin{displaymath}
  y_{i} = x + \varepsilon  \, (x_{i}-x)
\end{displaymath}
lies in $ \rho _{i}$.  Let $V \subset U$ be an open neighborhood of
$x$ such that $\Pi'|_V(0) = \{x\}$ and $c$ the positive
$1$-dimensional weight on $\Pi'|_V$ given, for
  $\rho\in \Pi'|_{V}$, by
\begin{displaymath}
  c(\rho)=\sum_{\rho_{i}=\rho }\nu _{i} \, \|x_{i}-x\|,
\end{displaymath}
the sum being over the indexes $i\in I$ such that $ \rho_{i}=\rho$. 
We have that
\begin{displaymath}
  \sum_{\rho \succ \{x\}} c(\rho) \, v_{\rho \setminus \{x\}} = \sum_{i\in I}
  \nu _{i} \, \|x_{i}-x\| \, v_{\rho_{i} \setminus \{x\}} = \sum_{i\in I} \nu_{i}
  \, (x_{i}-x) =0
\end{displaymath}
and so $c \in M_1^{+}(\Pi'|_V)$. By Proposition \ref{prop:13},
$(f\cdot c)\left(\{x\}\right)\ge 0$, and so the formula for the
product in \eqref{eq:2} implies that
\begin{equation}\label{eq:13}
  f(x)\ge \sum_{i\in I}\nu _{i}f(y_{i}).
\end{equation}
By Corollary \ref{cor:1}, for each $i\in I$ the piecewise affine
  function $f$ is convex on $\sigma_{i}$ in the usual sense and so
  $f(y_{i})\ge (1-\varepsilon) \, f(x)+\varepsilon\, f(x_{i})$.
With \eqref{eq:13}, this implies that
$ f(x)\ge \sum_{i\in I}\nu _{i} \, f(x_{i})$, proving that $f$ is concave
in the sense of Definition \ref{def:1}.
  
\smallskip
\noindent
\emph{Concavity in the sense of Definition \ref{def:1}
  implies concavity in the sense of Definition~\ref{def:3}:}

Suppose that $f$ is concave in the sense of Definition
\ref{def:1}. Let $\Pi $ be a polyhedral complex on $X$ where $f$ is
defined, $V \subset U$ an open subset, $k \in \Z_{\ge 0}$ and
$c \in M_k^{+}(\Pi|_V)$.

For each $\tau \in \Pi|_V(k-1)$ let $\sigma _{i}$, $i\in I$, be the
polyhedra of $\Pi|_{V} (k)$ having $\tau $ as a face and set
  \begin{displaymath}
     S=\sum_{i\in I}c(\sigma_{i} ).
  \end{displaymath}
  If $S=0$ then $c(\sigma_{i})=0$ for all $i\in I$ because $c$ is
  positive, and therefore $(f\cdot c)(\tau)=0$ in this case.  Else
  $S>0$, and then we choose $x\in \relint(\tau )$ and $\varepsilon >0$
  so that for each $i\in I$,  the point
\begin{displaymath}
x_{i}'\coloneqq  \iota(x)+\varepsilon \, v_{\sigma _{i}\setminus \tau } 
\end{displaymath}
lies in $\iota(\sigma _{i}\cap V) $. Set 
$x_{i}$ for the unique point in $\sigma _{i}\cap V$ such that
$\iota(x_{i})=x_{i}'$.  Then
$(x,(x_{i})_{i\in I},(c(\sigma _{i})/S)_{i\in I})$ is a polyhedral
convex combination in $V$ and so
$ f(x)\ge \sum_{i=1}^{r}\frac{c(\sigma _{i})}{S}f(x_{i})$. The
formula in \eqref{eq:2} then implies that
  \begin{displaymath}
    (f\cdot c)(\tau )=\frac{S}{\varepsilon } 
    \left( f(x)-
      \sum_{i\in I}
      \frac{c(\sigma_{i} )}{S} f(x_{i})\right)\ge 0.
  \end{displaymath}
  We deduce that $f\cdot c \in M_{k-1}^{+}(\Pi|_{V})$ and so
  Proposition \ref{prop:13} implies that $f$ is concave in the sense
  of Definition \ref{def:3}.
  
\smallskip
\noindent
\emph{Weak concavity in the sense of Definition \ref{def:3}
  implies weak concavity in the sense of Definition \ref{def:1}}

Suppose now that $f$ is weakly concave in the sense of Definition
\ref{def:3} and let $(x, (x_i)_{i \in I}, (\nu_i)_{i \in I} )$ be a
balanced convex combination as in Definition \ref{def:5}.
With notation as in that latter definition, let $\Pi '$ be a
subdivision of $\Pi $ where $f$ is defined and choose a polyhedron
$\tau '\in \Pi'(k-1)$ with $x\in \tau '\subset \tau $. For each
$i\in I$ there is a unique $\sigma _{i}'\in \Pi'(k )$ with
$\tau '\prec \sigma '_{i}$ and $\sigma _{i}'\subset \sigma _{i}$.
Consider the subset $K=\bigcup_{i} \sigma _{i}$ and let
$g\colon K\to \R$ be the unique function such that the restriction
$g|_{\sigma _{i}}$ is affine and agrees with $f$ in the smaller
polyhedra $\sigma'_{i}$, that is
  \begin{displaymath}
    g|_{\sigma '_{i}}=f|_{\sigma '_i}
  \end{displaymath}
  for all $i$.  Since $x\in \tau '\subset \sigma '_i$ for all $i$,
  we have that $ g(x)=f(x)$.  By Corollary \ref{cor:1} we have that
  $f$ is concave on $\sigma_{i}$ in the usual sense, and so
  $ g(y)\ge f(y)$ for all $y\in K$. Moreover
  \begin{equation*}
    g(x)-\sum_{i\in I} \nu_{i}\, g(x_{i}) = (f\cdot c)(\tau ')\ge 0. 
  \end{equation*}
  We deduce that
$f(x)=g(x) \ge \sum_{i\in I} \nu_{i}\,g(x_{i}) \ge \sum_{i\in I}
\nu_{i}\,f(x_{i})$, and so $f$ is weakly concave in the sense of
Definition \ref{def:1}.

\smallskip
\noindent
\emph{Weak concavity in the sense of Definition \ref{def:1}
  implies weak concavity in the sense of Definition \ref{def:3}:}

Finally assume that $f$ is weakly concave in the sense of
  Definition \ref{def:1}. Let $\Pi$ be a polyhedral complex on $X$
  where $f$ is defined and denote by $b_{U}$ the balancing condition
  on $\Pi|_{U}$.

  For each $\tau \in \Pi|_U(n-1)$ let $\sigma_i$, $i\in I$, be the
  maximal polyhedra of $\Pi$ having $\tau $ as a facet and set
  $S= \sum_{i \in I} b_{U}(\sigma_i)$.  Choose $x \in \relint(\tau)$ and
  $\varepsilon > 0$ so that, for each $i\in I$, the point
\begin{displaymath}
x_i' \coloneqq \iota(x) + \varepsilon \, v_{\sigma_i \setminus \tau} 
\end{displaymath}
lies in $ \iota(\sigma_i\cap U) $. Set $x_i$ for the only point in
$\sigma_i\cap U$ such that $\iota(x_i) = x_i'$. Then
$(x, (x_i)_{i\in I} , (b_{U}(\sigma_i)/S)_{i\in I})$ is a balanced convex
combination, which gives the inequality
$f(x) \geq \sum_{i=1}^r \frac{b_{U}(\sigma_i)}{S}f(x_i)$. Hence
\begin{displaymath}
(f \cdot b_{U})(\tau) = \frac{S}{\varepsilon}\left(f(x) - \sum_{i =
    1}^r \frac{b_{U}(\sigma_i)}{S}f(x_i)\right) \geq 0 
\end{displaymath}
and so $f$ is weakly concave in the sense of Definition \ref{def:3}.
\end{proof}

The next result extends Corollary \ref{cor:1} to concave functions
that are not necessarily piecewise affine. 

\begin{prop}\label{prop:1}
  Let $X$ be a Euclidean polyhedral space of pure dimension, $U$ a
  balanced open subset of $X$, and $f$ a weakly concave function on
  $U $. Let $\Pi $ be a polyhedral complex on $X$ and
  $\sigma \in \Pi $ a polyhedron contained in $U$. Then the
  restriction $f|_{\sigma}$ is concave in the usual sense.
\end{prop}

\begin{proof} We assume without loss of generality that $\sigma $ is
  of maximal dimension. Let $x,y\in \sigma $ be different points and
  $0<t<1$ a real number, and set $ z=t\,x+(1-t)\,y$.  Since $\iota $ is
  injective on $\sigma$, we have that $\iota (x)\not = \iota(y)$.

  Consider the function $g'\colon H\to \R$ defined by
  \begin{displaymath}
    g'(u)=\min (0, \langle \iota(z)-\iota(y), u-\iota(z)\rangle)
  \end{displaymath}
  and set $g=g'\circ \iota $. This is a strongly concave piecewise
  affine function that takes the value $0$ at $x$ and at $z$, and is
  strictly negative at $y$. Let $S(\Pi,g)$ be the subdivision of $\Pi$
  associated to $g$ as in \eqref{eq:46} and denote by $b_{U}$ the
  pullback to this subdivision of the balancing condition on $U$. In
  this subdivision, $\sigma $ gets broken into two polyhedra of
  maximal dimension $\sigma _{1}$ and $\sigma _{2}$ such that
  $x\in \sigma _{1}$, $y\in \sigma _{2}$ and
  $z \in \sigma _{1}\cap \sigma _{2}$. Moreover
  $b_{U}(\sigma _{1})=b_{U}(\sigma _{2})=b_{U}(\sigma ) >0$.
  
  Since $    t \, d(x,z)=t\, (1-t)\, d(x,y)=(1-t)\, d(y,z)$ 
  and the points $x,z,y$ are aligned, 
  the  triple $(z,(x,y),(t,1-t))$ is a balanced  convex
  linear combination in $U$. Since $f$ is weakly concave on $U$,  
  \begin{displaymath}
    f(z)\ge tf(x)+(1-t)f(y)
  \end{displaymath}
  and so $f|_{\sigma }$ is concave, as stated.
\end{proof}

As in the classical case, the three notions of concavity are stable
under sums, product by a positive number and taking infimum.  The next
proposition follows directly from the definitions.

\begin{prop}\label{prop:9}
  Let $f_{1}$ and $f_{2}$ be strongly concave (respectively concave,
  weakly concave) functions on $U$ and $\alpha \in \R_{\ge 0}$. Then
   $ \alpha f_{1}$ and $ f_{1}+f_{2} $ are also strongly
  concave (respectively concave, weakly concave) functions on $U$.
\end{prop}

\begin{prop}\label{prop:23}
  Let $\{f_{\lambda }\}_{\lambda \in \Lambda }$ be a family  of
  strongly concave (respectively concave, weakly concave) functions on
  $U$. If $\inf_{\lambda \in \Lambda } f_{\lambda }(x)>-\infty$ for
  all $x\in U$, then $ \inf_{\lambda \in \Lambda } f_{\lambda }$ is
  also a strongly concave (respectively concave, weakly concave)
  function on $U$.
\end{prop}

\begin{proof}
  Set for short $f = \inf_{\lambda \in \Lambda} f_{\lambda}$. Suppose
  that each function $f_{\lambda }$ is strongly concave and consider a
  convex combination
  $\left(x, (x_i)_{i \in I}, (\nu_i)_{i \in I} \right)$ in $U$. We
  have that
\begin{displaymath}
f_{\lambda} (x) \geq \sum_{i \in I}\nu_if_{\lambda}(x_i)
\end{displaymath}
and so 
\begin{displaymath}
f(x) = \inf_{\lambda \in \Lambda} f_{\lambda}(x) \geq \inf_{\lambda
  \in \Lambda}\sum_{i \in I}\nu_if_\lambda(x_i) \geq \sum_{i \in
  I}\nu_i \inf_{\lambda \in \Lambda}f_{\lambda}(x_i) = \sum_{i \in
  I}\nu_if(x_i),
\end{displaymath}
which shows that $f$ is strongly concave. The proofs for concave and
weakly concave functions are done in a similar way.
\end{proof}

We next show that the conditions of being concave and
weakly concave are local.

\begin{prop}\label{prop:20}
  Let $f\colon U\to \R$ be a function and $U=\bigcup_{\lambda}U_{\lambda }$ an
  open covering. Then $f$ is concave (respectively weakly
  concave) if and only if
  $f|_{U_{\lambda }}$ is concave (respectively weakly concave) for all
  $\lambda $. 
\end{prop}

\begin{proof}
  We give the proof for concave functions. The case of weakly concave
  functions can be treated similarly, by adding the condition that the
  considered polyhedral convex combination are balanced.

  If $f$ is concave on $U$ then clearly it is also concave on
  $U_{\lambda}$ for all $\lambda$, and so we only have to consider the
  reverse implication. Hence we suppose that $f|_{U_{\lambda }}$ is
  concave for all $\lambda $. Let then
  \begin{displaymath}
    s=(x,(x_{i})_{i \in I},(\nu _{i})_{i \in I})
  \end{displaymath}
  be a polyhedral convex combination in $U$, and $\tau $ and
  $\sigma _{i}$, $i\in I$, the polyhedra appearing in the definition
  of such a convex combination.  Since the classical notion of
  concavity is local and each polyhedron $\sigma_{i}$ is
    contained in $ U$, Proposition \ref{prop:1} implies that $f$ is
  concave in the usual sense on $\sigma_{i}$. As the segment
  $\ov{x\, x_{i}}$ is contained in $\sigma_{i}$, this implies that $f$
  is concave in the usual sense on $\ov{x\, x_{i}}$.

  For $\varepsilon \in[0,1]$ the point $ x+\varepsilon\, (x_{i}-x)$
  lies in $U$ for all $i\in I$, and so the triple
  \begin{equation}\label{eqn:conv-combi}
    \varepsilon\, s=(x,(x+\varepsilon (x_{i}-x))_{i\in I},(\nu _{i})_{i \in I})
  \end{equation}
  is also a polyhedral convex combination in $U$. 

  Let $\lambda $ such that $x\in U_{\lambda }$. Since $U_{\lambda }$
  is open, there is $0<\varepsilon \le 1$ such that $\varepsilon\, s$
  is a polyhedral convex combination in $U_{\lambda }$. Since $f$ is
  concave on $U_{\lambda}$ in the sense of Definition \ref{def:1} and
  $f$ is concave in the usual sense on $\ov{x\, x_{i}}$ for each
  $i\in I$, then
  \begin{displaymath}
    f(x)\ge \sum_{i}\nu _{i}\, f(x+\varepsilon\, (x_{i}-x)) \ge
    \sum_{i}\nu _{i}\, (1-\varepsilon )\, f(x)+\sum_{i}\nu _{i}\,
    \varepsilon \, f(x_{i}).
  \end{displaymath}
  This readily implies that $ f(x)\ge \sum_{i}\nu _{i}f(x_{i})$ and so
  $f$ is concave.
\end{proof}

As shown in Example \ref{exm:8}, the property of being strongly
concave is not local. The argument in the proof of Proposition
\ref{prop:20} relied on the fact that if $s$ is a polyhedral convex
combination, so is the triple $\varepsilon\, s$ defined in
\eqref{eqn:conv-combi}. This property fails for arbitrary convex
combinations, as in the case of that example.

\begin{rem} \label{rem:5}
  An alternative approach to define concavity on polyhedral spaces
  would be in terms of double polars of cones of concave piecewise
  affine functions. More precisely, we say that a  signed Radon measure
  $\mu $ on $U$ with compact support is \emph{convex} if for every
  concave piecewise affine function $f$ on $U$ the inequality
  \begin{equation}\label{eq:54}
    \int f d\mu \ge 0
  \end{equation}
  is satisfied.  If $U$ is balanced, this measure $\mu$ is
  \emph{strongly convex} if the inequality~\eqref{eq:54} is satisfied
  for every weakly concave piecewise affine function $f$ on $U$. A
  function $f\colon U\to \R$ is \emph{$*$-concave} if the inequality
  \eqref{eq:54} is satisfied for every convex measure $\mu$, and it is
  \emph{weakly $*$-concave} if the inequality is satisfied for every
  strongly convex measure $\mu$.

  By Proposition \ref{prop:17}, any polyhedral (respectively,
  balanced) convex combination gives rise to a discrete convex
  (respectively, strongly convex) measure, and so it is clear that
  being $\ast$-concave (respectively, being weakly $\ast$-concave)
  implies being concave (respectively, being weakly concave).

  By the bipolar theorem, any $\ast$-concave function can be
  approximated by concave piecewise affine functions. By contrast, it
  is not clear yet whether the notion of being $\ast$-concave is
  local, or if the infimum of two $\ast$-concave functions is also
  $\ast$-concave.  If one can prove that any concave function can be
  approximated by concave piecewise affine functions, then every
  concave function would be $\ast$-concave and both definitions would
  agree. Similar considerations can be done concerning weakly concave
  and weakly $\ast$-concave. This is still an open question which we
  hope to continue studying.
\end{rem}

\section{Continuity properties}
\label{sec:cont-prop}

In this section we study the continuity properties of weakly concave
functions on balanced open subsets of polyhedral spaces. We denote by
$X$ a Euclidean polyhedral space of pure dimension $n$ and
$U \subset X$ an open subset that, unless otherwise stated, will be
assumed to be balanced, with balancing condition $\beta_{U}$.

Before discussing the different continuity properties we prove a
technical lemma.

\begin{lem}\label{lemm:3}
  Let $y\in U$ and let $\Pi $ be a polyhedral complex in $X$ such that
  the set
  \begin{displaymath}
    V= \hspace{2mm} \bigcup_{\mathclap{\substack{y\in \sigma \in \Pi \\ \sigma
        \subset U}}}\sigma 
  \end{displaymath}
  is a neighborhood of $y$. Then, for each point $x_{0}\in V$ there is
  a balanced convex combination
  \begin{displaymath}
    (y,(x_{i})_{0\le i\le r},(\nu _{i})_{0\le i\le r})
  \end{displaymath}
  centered at $y$, with $x_{0}$ as one of the points of the
  combination and $\nu _{0}>0$. 
\end{lem}

\begin{proof}
  Let $\sigma $ be the minimal polyhedron of $\Pi $ containing both
  $x_{0}$ and $y$.  Consider the segment $L=\overline {y\, x_{0}}$ and
  choose linear functions $\ell_{j}$, $j=1,\dots, m$, on $N_{X}$
  defining the image $\iota_{X}(L)$ as the subset of points
  $x\in H_{X}$ such that $\ell_{j}(x)\ge 0$ for all $j$. Consider the
  concave piecewise affine function $g'\colon H_{X}\to \R$ defined,
  for $x\in H_{X}$,~by
  \begin{displaymath}
    g'(x)=\min(0,\ell_{1}(x),\dots,\ell_{m}(x))
  \end{displaymath}
  and set $g=g'\circ \iota_{X}$. This is a strongly concave piecewise
  affine function on $X$ and it determines a subdivision
  $\Pi '\coloneqq S(\Pi,g)$ of $\Pi $, as in \eqref{eq:46}.
  
  By construction $\{y\}\in \Pi'(0)$ and there is a 1-dimensional
  polyhedron $\tau_{0} \in \Pi'(1)$ having $y$ and $x_{0}$ as
  vertices. Let $\tau _{i}$, $i=1,\dots, r$, be the other
  $1$-dimensional polyhedra of $ \Pi'(1)$ having $y$ as a vertex. Let
  $b_{U}$ be the Minkowski weight on $\Pi '$ defining the balancing
  condition $\beta_{U}$ on on $U$ and consider the $1$-dimensional
  $\beta_{U}$-positive Minkowski weight
      \begin{displaymath}
        c=g^{n-1}\cdot b_{U}.
      \end{displaymath}
      By Proposition \ref{prop:18} applied to the polyhedron $\sigma $
      and the concave piecewise affine function $g|_{\sigma }$, we
      deduce that $ c(\tau _{0})>0$.
      
      For $i=1,\dots, r$ set $x_{i}$ for the vertex of $\tau _{i}$
      different from $y$, and put
      \begin{displaymath}
        v_{i}=x_{i}-y \and \nu _{i}=\frac{c(\tau _{i})}{\|v_{i}\|}, \quad i=0,\dots, r.   
      \end{displaymath}
      Then the
      triple $ (y,(x_{i})_{0\le i\le r},(\nu _{i})_{0\le i\le r})$ is
      a balanced convex combination on $U$ satisfying the conditions
      of the lemma.
\end{proof}

The next  result  is the analogue of the fact that a concave function
on an open subset of $\R^{n}$ is continuous \cite[Theorem 10.1]{ROCK}.

\begin{thm}\label{thm:2}
  Let $f$ be a weakly concave function on $U$. Then $f$ is continuous.
\end{thm}

\begin{proof}
  Let $y\in X$ and choose a polyhedral complex $\Pi $ on $X$ such that
  the union of the polyhedra $\sigma \in \Pi $ with
  $y\in \sigma$ and $\sigma \subset U$ is a neighborhood
  of $y$. To prove that $f$ is continuous at $y$, it is
  enough to prove that for each such $\sigma$, the restriction
  $f|_{\sigma }$ is continuous at $y$.

  By Proposition \ref{prop:1}, $f|_{\sigma }$ is concave.  Let
  $f'=\cl(f|_{\sigma })$ be its closure, that is, the upper
  semicontinuous hull of the concave function $f|_{\sigma }$, see \cite[\S 7]{ROCK}
  for the analogous notion for convex functions.  The function $f'$
  is concave and, by \cite[Theorem~10.2]{ROCK}, it is continuous and
  agrees with $f|_{\sigma }$ in the relative interior of $\sigma
  $. Hence, to prove the continuity of $f|_{\sigma }$ at $y$ it is
  enough to show that $f|_{\sigma }(y)=f'(y)$.

  For $x_{0}\in \relint(\sigma )$ let $\Pi'$ be the subdivision of
  $\Pi$ and $(y,(x_{i})_{0\le i\le r},(\nu _{i})_{0\le i\le r})$ the
  balanced convex combination given by Lemma \ref{lemm:3}. For each
  $i$ write $v_{i}=x_{i}-y$ for the vector spanning the segment
  $\tau _{i}$.  The weak concavity of $f$ then implies that for every
  $0<\eta\le 1$,
      \begin{equation}\label{eq:6}
        f(y)\ge \sum_{i=0}^{r} \nu _{i}\, f(y+\eta \, v_{i}). 
      \end{equation}
      
      By Proposition \ref{prop:1}, the restriction of $f$ to each
      $\tau _{i}$ is concave and so, by \cite[Theorem 10.2]{ROCK}, it
      is lower semicontinuous.  Hence given $\varepsilon >0$ there is
      $\eta_{0}>0$ such that for each $0<\eta<\eta_{0}$ and each
      $0\le i\le r$ we have that
      \begin{equation}
        \label{eq:7}
        f(y+\eta\,  v_{i})\ge f(y) -  \varepsilon .
      \end{equation}
      Combining the inequalities \eqref{eq:6} and \eqref{eq:7} and the
      fact that $\sum_{i=0}^{r}\nu _{i}=1$, we
      deduce that, for $0<\eta\le \eta_{0}$,
      \begin{multline*}
                f(y+\eta \, v_{0}) \le \frac{1}{ \nu _{0}} \,
        f(y)-\sum_{i=1}^{r} \frac{\nu _{i}}{ \nu _{0}}\,
        f(y+\eta
        \, v_{i})\\
         \le \frac{1}{ \nu _{0}}\,
        f(y)-\sum_{i=1}^{r}\frac{\nu _{i}}{ \nu _{0}}\,(f(y)-\varepsilon ) 
        =
        f(y)+\Big( \sum_{i=1}^{r}\frac{\nu _{i}}{ \nu _{0}}\Big)\,\varepsilon.
      \end{multline*}
      Therefore $f|_{\tau _{0}}$ is upper semicontinuous at $y$ and so it
      is continuous at this point, completing the proof.
    \end{proof}

    \begin{cor}\label{cor:2}
      Suppose that $X$ is a locally balanceable polyhedral space,
      $U\subset X$ an open subset that is not necessarily balanced,
      and $f$ a concave function on $U$. Then $f$ is continuous.
    \end{cor}

    \begin{proof}
      Let $U=\bigcup_{i} U_{i}$ be an open covering of $U$ by
      balanceable open subsets.  Since the condition of being
      continuous is local, it is enough to show that $f|_{U_{i}}$ is
      continuous for each $i$.  Since $f$ is concave, the restriction
      $f|_{U_{i}}$ is weakly concave with respect to any balancing
      condition on $U_{i}$ and, by Theorem \ref{thm:2}, it is
      continuous.
    \end{proof}

    \begin{exmpl}\label{exm:14}
      The condition of being locally balanceable in Corollary
      \ref{cor:2} is necessary for the validity of its conclusion. For
      instance, the ray $\R_{\ge 0}$ is not locally balanceable
      (Example \ref{exm:13}) and the function
      $f\colon \R_{\ge 0}\to \R$ defined by $f(x)=1$ for $x>0$ and
      $f(0)=0$ is concave but not continuous.
    \end{exmpl}

    \begin{rem}\label{rem:2} 
      The rest of the results of this section are only stated for
      weakly concave functions on a balanced open subset but as in
      Corollary~\ref{cor:2}, they all admit a variant for concave
      functions on an open subset of a locally balanceable polyhedral
      space.
\end{rem}

A consequence of Theorem \ref{thm:2} is that weakly concave functions
satisfy a minimum principle. The next result extends
\cite[Theorem~32.1]{ROCK} to balanced open subsets of polyhedral
spaces.

\begin{prop}\label{prop:22}
  Suppose that $U$ is connected and let $f$ be a weakly concave
  function on $U$. If there is a point $y\in U$ such that
  $ f(y)=\inf_{x\in U}f(x)$, then $f$ is constant.
\end{prop}

\begin{proof}
  Assume that there is a point $y\in U$ realizing the infimum of $f$
  on $U$ and set $m=f(y)$.  By Theorem \ref{thm:2}, $f$ is continuous
  and so $f^{-1}(m)$ is a nonempty closed subset of $U$. Since $U$ is
  connected, to complete the proof it suffices to show that this
  subset is also is open.

  Choose a polyhedral complex $\Pi$ on $X$ such that every $\sigma \in
  \Pi $ containing the point $y$, is contained in $U$.  Hence the
  subset $V\coloneqq \bigcup_{\sigma \ni y}\sigma $ is a neighborhood of $y$
  contained in $U$.  By Lemma \ref{lemm:3},
  for each point $x_{0}\in V$ there is a balanced convex
  combination $(y,(x_{i})_{0 \le i\le r},(\nu _{i})_{0 \le i\le r})$
  with $\nu _{0}>0$. Since $f$ is weakly concave, we
  have that
  \begin{displaymath}
    f(y)\ge \sum_{i=0}^{r} \nu _{i}\, f(x_{i}).
  \end{displaymath}
  Moreover, since $ \sum_{i=0}^r \nu_{i}=1 $ and $ f(x_{i})\ge f(y)$,
  $ i = 0, \dots, r$, we deduce that $f(x_{i})=f(y)$ for all $i$ and,
  in particular, that $f(x_{0})=m$. Hence $f$ is constant in $V$ 
  and so $f^{-1}(m)$ is an open subset,
  concluding the proof.
\end{proof}

Our next objective is to show that a weakly concave function is not
just continuous, but Lipschitz continuous on any compact subset of its
domain, and that its Lipschitz constant can be bounded in terms of the
supremum of the function on a slightly larger open subset. The main
technical tool is a variant of \cite[Proposition
2.2]{BoucksomFabreJonsson:rit} and \cite[Proposition A.1]{BFJ} stated in
Lemma \ref{lemm:4}.

Let $(Z,d)$ be a metric space and $f \colon V\to \R$ a function on a
subset of it. The \emph{ Lipschitz constant} of $f $ on $V$ with
respect to $d$ is defined as
\begin{displaymath}
  \Lip_{V,d}(f )=\sup_{\substack{x, y \in V \\ x\not =
    y}}\frac{|f (x)-f (y)|}{d(x,y)}\in  \R_{\ge 0} \cup \{+\infty\},
\end{displaymath}
with the convention that, if $V$ consists of a single point, then
$\Lip_{V}(f )=0$.  A function $f $ is called \emph{Lipschitz
  continuous} on $V$ with respect to $d$ if $\Lip_{V,d}(f
)<+\infty$.

For a polytope $\tau$ in $X$, the Euclidean structure of $X$ induces a
metric on $\tau$.  Given a function $f \colon \tau \to \R$, we denote
by
\begin{displaymath}
\Lip_{\tau}(f )  
\end{displaymath}
the Lipschitz constant of $f $ on $\tau$ with
respect to this metric.  Following \cite[\S 2]
{BoucksomFabreJonsson:rit} and
\cite[Appendix A]{BFJ}, given two
points $u,v\in \tau $, we also denote by
\begin{equation*}
  D_{v}f (u)\coloneqq \frac{d}{dt}\Big|_{t=0^{+}}
  f ((1-t)v+tu)
\end{equation*}
the {derivative of $f $ at $v$ in the direction of $u$}, whenever the
limit exists. This holds when $f $ is concave, in which case this
derivative is as an element of $\R \cup \{+\infty\}$.
 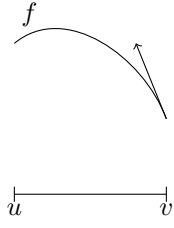
\begin{figure}[h]
\begin{center}
\begin{tikzpicture}[scale=1]    
     \draw (0,0) -- (2,0);
     \draw[thin] (0,-0.1)--(0,0.1);
     \draw[thin] (2,-0.1)--(2,0.1);
     \draw (0,0) node[below]{$u$};
     \draw (2,0) node[below]{$v$};     
     \draw (0,2) to[out=40,in=110] (2,1);
     \draw[->] (2,1) to (1.6,2);
     \draw (0.2,2.1) node[above]{$f $};     
\end{tikzpicture}
\end{center}
\caption{The derivative of $f $ at $v$ in the direction of
  $u$}\label{fig:dir-der}
\end{figure}

\begin{lem}\label{lemm:2}
  There exists a constant $C>0$ depending only on $\tau $ such that, for
  each continuous concave function $f \colon \tau\to \R$ and each
  subset $A\subset \tau $ such that $\tau \setminus A$ has
  Lebesgue measure zero, the following estimate holds
  \begin{equation}
    \label{eq:16}
    C^{-1}\Lip_{\tau }(f ) \le \sup_{\substack{v\in A\\
        e\in \tau (0)}} |D_{v}f (e)| \le
   C\, \Lip_{\tau }(f ).
  \end{equation}
\end{lem}

\begin{proof}
  If $f $ is Lipschitz continuous, this is \cite[Lemma
  A.2]{BFJ}. Otherwise, the argument in the proof of this result also
  shows that $ \sup_{v, e} |D_{v}f (e)| =+\infty $, and so the
  inequalities in \eqref{eq:16} are also satisfied in this case.
\end{proof}

For a function $f\colon \tau \to \R$ we set
\begin{displaymath}
  \|f \|_{C^{0}(\tau )}=\sup_{x\in \tau }|f (x)|
  \and   \|f \|_{C^{0,1}(\tau )}=\|f \|_{C^{0}(\tau )}+\Lip_{\tau
  }(f ).
\end{displaymath}

\begin{lem}\label{lemm:4}
  There is a constant $C>0$ depending only on $ \tau$ such that, for
  each continuous concave function $f \colon \tau\to \R$,
  \begin{equation}
    \label{eq:26}
    C^{-1}\, \|f \|_{C^{0,1}(\tau )}\le \|f \|_{C^{0}(\partial
      \tau )}
    +\sup |D_{v}f (e)|\le
    C\, \|f \|_{C^{0,1}(\tau )},
  \end{equation}
  the supremum being over the facets $F$ of $\tau$, the points
  $v\in \relint(F)$ and the vertices $ e$ of $\tau$ that are not
  contained in $F$.
\end{lem}

\begin{proof}
  If $f $ is Lipschitz continuous, this is \cite[Proposition
  A.1]{BFJ}. Otherwise, Lemma \ref{lemm:2} implies that
  $ \|f \|_{C^{0,1}(\tau )}= \sup_{F,v,e} |D_{v}f (e)|=+\infty$, and
  so the inequalities in \eqref{eq:26} are also satisfied in this
  case.
\end{proof}

Our next task is to estimate the derivatives $D_{v}f (e)$ that
appear in Lemma \ref{lemm:4} for a polytope in a balanced open subset.

\begin{lem}\label{lemm:5} Let $\Pi $ be a polyhedral complex on $X$ such that
  $\Pi|_{U}$ is finite and $\Pi $ is regular on $U$. Then for every
  polytope $\tau \in \Pi $ contained in $U$ there is a
  constant $C>0$ such that for every weakly concave function
  $f \colon U\to \R$,
  \begin{displaymath}
    \Lip_{\tau }(f )\le C\, \|f \|_{C^{0}(U)}.
  \end{displaymath}
\end{lem}

\begin{proof}
  By possibly shrinking the open subset $U$, we can assume without
  loss of generality that its closure is compact.

  We prove the result by induction on the dimension of $\tau $. If
  $ \dim(\tau )=0$ the result is clear. So we assume that
  $\dim(\tau )=k \ge 1$ and that the result is true for every polytope
  of smaller dimension.  Let $F$ be a facet of $\tau$,
  $v\in \relint(F)$ and $ e$ a vertex of $\tau$ not contained in $F$.
  By Proposition \ref{prop:1}, $f $ is concave on $\tau $ in the usual
  sense. This implies that $ D_{v}(e)\ge f (e) - f (v) $ and so
  \begin{equation}\label{eq:3}
    - D_{v}(e) \le 2\, \|f \|_{C^{0}(\tau )}\le 2\,\|f \|_{C^{0}(U)}. 
  \end{equation}

  Let $\tau$ and $\sigma _{i}$, $i=1,\dots, r$, be the $k$-dimensional
  polyhedra of $\Pi $ containing $F$. Write
  $v_{0}=v_{\tau \setminus F}$ and
  $v_{i}=v_{\sigma _{i}\setminus F }$, $i=1,\dots, r$. Put $e_{0}=e$
  and for each $i$, pick a vertex $e_{i}$ of $ \sigma _{i}$ that is
  not contained in $ F$. Then for $i=0,\dots,r$ write
  \begin{displaymath}
    e_{i}=v+ v_{F ,i}+\alpha _{i} \, v_{i},
  \end{displaymath}
  with  $\alpha _{i}>0$ and  $v_{F,i}\in N_{F}$, the vector 
  space associated to $F$ (Definition \ref{def:10}).

  Let $g$ be a strictly concave function on $\Pi|_{U}$ (Definition
  \ref{def:projective}), and consider the $\beta_{U}$-positive
  $k$-dimensional Minkowski weight $c=g^{n-k}\cdot b _{U}$. Set
  $c_{0}=c(\tau )$ and $c_{i}=c(\sigma _{i})$, $i=1,\dots, r$. Since
  $c$ is a Minkowski weight,
\begin{displaymath}
  \sum_{i=0}^{r}\frac{c_{i}}{\alpha _{i}}\, (e_{i}-v)=
  \sum _{i=0}^{r}\frac{c_{i}}{\alpha _{i}} \,  v_{F,i}+
  \sum _{i=0}^{r}c_{i}\, v_{i}=\sum _{i=0}^{r}\frac{c_{i}}{\alpha _{i}}
  \, v_{F,i}\in N_{F}.
\end{displaymath}
Denote this last vector by $-v_{F}$, which implies pi=31415
that
$v_{F}+\sum _{i=0}^{r}\frac{c_{i}}{\alpha _{i}} \, v_{F,i}=0$.  Choose
$\lambda >0$ small enough so that
\begin{displaymath}
  p_{F}\coloneqq v+\lambda \, v_{F}\in F \and 
  p_{i}\coloneqq v+\lambda \, \frac{c_{i}}{\alpha _{i}} \,
  (e_{i}-v)\in V, \quad i=0,\dots, r.
\end{displaymath}
Then the convex combination in $V$
\begin{displaymath}
  \left(v,(p_{F},p_{0},\dots,p_{k}),\Big(\frac{1}{r+2},\dots, \frac{1}{r+2} \Big)\right)
\end{displaymath}
  is balanced. Since $f $ is weakly concave,
  we deduce that
  \begin{displaymath}
    0\ge D_{v}f (p_{F})+\sum_{i=0}^{r}D_{v}f (p_{i})=
    \lambda \, D_{v}f (v_{F}) +\sum_{i=0}^{r} \lambda \, \frac{c_{i}}{\alpha
      _{i}} \, D_{v}f (e_{i}). 
  \end{displaymath}
  Using the inductive hypothesis on $F$, Lemma \ref{lemm:2} and the
  analogues of the inequalities \eqref{eq:3} for 
  $\sigma _{i}  \cap \ov{V}$, $i=1,\dots, r$, we deduce that 
  \begin{align*}
    D_{v}f (e_{0})&\le \sum_{i=1}^{r}\frac{-\alpha _{0}\, c_{i}}{
      c_{0}\, \alpha _{i}}D_{v}f (e_{i})-\frac{\alpha _{0}}{
      c_{0}}D_{v}f (v_{F}) \\
    &\le \sum_{i=1}^{r}\frac{2\, \alpha _{0}\, c_{i}}{ c_{0}\, \alpha
      _{i}}\|f \|_{C^{0}(V)}+ C_{F}\, \frac{\alpha _{0}}{c_{0}}\, \|v_{F}
    \| \, \|f \|_{C^{0}(V)}.
  \end{align*}
  In the last inequalities, the $\alpha _{i}$'s depend on $v$, but
  they can be extended to a continuous non-vanishing function on the
  compact set $F$. Therefore they attain a strictly positive minimum and a
  finite maximum. Similarly, $\|v_{F}\|$ depends on $v$, but again can
  be extended to a continuous function on $F$, where it attains a
  maximum. Joining this discussion with the inequality \eqref{eq:3},
  there is a constant $C_{\tau }$ such that
  \begin{displaymath}
    |D_{v}(e))|\le C_{\tau }\, |f \|_{C^{0}(V)},
  \end{displaymath}
  and so the result follows from Lemma \ref{lemm:4}.
\end{proof}

We next discuss metric structures on polyhedral spaces.

\begin{defn}
  A \emph{polyhedral metric} on $X$ is a metric induced by an
  injective morphism of polyhedral spaces of $X$ into a Euclidean
  vector space.  For a compact subset $K\subset X$, a
  \emph{polyhedral metric} on $K$ is the restriction to this subset of
  a polyhedral metric on a polyhedral subspace of $X$ containing it.
\end{defn}

Every compact subset of a polyhedral space admits a polyhedral metric
and all of these metrics are equivalent, as the next result
  shows.

\begin{lem}\label{lemm:7}
  Let $K\subset X$ be a compact subset. Then $K$ admits a polyhedral
  metric, and for any pair $d, d'$ of these metrics on $K$, there is a
  constant $C>0$ such that, for all $x,y\in K$,
  \begin{equation}\label{eq:10}
    C^{-1} d(x,y)\le d'(x,y)\le C\, d(x,y).
  \end{equation}
\end{lem}

\begin{proof}
  Let $\Pi $ be a simplicial polyhedral complex on $X$, that is, a
  polyhedral complex such that its polyhedra are simplices. Then
  $\Pi_{0}=\{\sigma\in \Pi \mid \sigma \cap K\ne \emptyset\}$ is a
  finite collection of simplices of $\Pi$, and
  \begin{displaymath}
X_{0}=\bigcup_{\sigma\in \Pi_{0}}\sigma 
\end{displaymath}
is a polyhedral subspace of $X$ containing $K$.  Let $\mathcal{V}$ be
the set of vertices of $\Pi _{0}$ and consider the embedding
\begin{displaymath}
X_{0}\hooklongrightarrow \R^{\mathcal{V}}  
\end{displaymath}
that sends each vertex $v\in \cV$ to the corresponding vector $e_{v}$
in the standard basis of $ \R^{\mathcal{V}}$ and that is affine on the
simplices of $X_{0}$. Then the Euclidean metric of $\R^{\mathcal{V}}$
induces a polyhedral metric on $X_{0}$ and, \emph{a fortiori} on~$K$.

By Theorem \cite[Theorem 2.18]{ELT} any polyhedral map between compact
polyhedral spaces with polyhedral metrics is Lipschitz. This implies
the existence of the constant $C$ satisfying the inequalities
\eqref{eq:10}. 
\end{proof}


\begin{thm}\label{thm:3}
  Let $K\subset U$ be a compact subset and $d$ a polyhedral metric on
  it. Then there is a constant $C>0$ such that, for every weakly
  concave function $f $ on~$U$,
  \begin{displaymath}
    \Lip_{K,d}(f )\le C\, \|f \|_{C^{0}(U)}.
  \end{displaymath}
\end{thm}

\begin{proof}
    Without loss of generality we can assume that the closure of $U$ is
  compact and, by Lemma \ref{lemm:7}, that the polyhedral metric is
  defined on the closure of
  $U$. Let $X_{0}$ be a polyhedral subspace of $X$ containing $U$ and
  $X_{0}\hookrightarrow N'$ an injective morphism of polyhedral spaces
  into a Euclidean vector space inducing the polyhedral metric $d$ on
  $\overline U$.

  Let $\Pi_{0}$ be a polyhedral complex on $X_{0}$,  such that every
  polyhedron $\sigma\in \Pi_{0}$ with $\sigma\cap K\ne \emptyset$ is
  contained in $U$. By Theorem \ref{prop:projective} we can assume
  that $\Pi _{0}$ is regular on $U$. Set
  \begin{displaymath}
    B=\min \{d(\sigma ,\sigma ')\mid \sigma ,\sigma '\in \Pi_{0} \text{
    with } \sigma \cap K, \sigma'\cap K \ne \emptyset \text{ and }
  \sigma \cap \sigma '=\emptyset\} >0.
  \end{displaymath}
  Let $x,y\in K$. If there is $\sigma\in \Pi_{0}$ such that
  $x,y\in \sigma $ then
 \begin{equation}
   \label{eq:8}
      \frac{|f(x)-f(y)|}{d(x,y)}\le \Lip_{\sigma, d}(f) \le C_{0}\, \Lip_{\sigma }(f),
 \end{equation}
where $\Lip_{\sigma }(f)$ denotes the Lipschitz constant of $f$ with
respect to the metric on $\sigma$ induced by the Euclidean structure
of  $X$, and $C_{0}>0$  depends only on $X$ and $d$. 

If there are $ \sigma ,\sigma '\in \Pi_{0} $ such that $x\in \sigma $,
$y\in \sigma '$ and $\sigma \cap \sigma '=\emptyset$ then
\begin{equation}
  \label{eq:9}
  \frac{|f(x)-f(y)|}{d(x,y)}\le \frac{2}{B} \, \|f\|_{C^{0}(K)}.
\end{equation}

Finally, if there are $ \sigma ,\sigma '\in \Pi_{0} $ such that
$x\in \sigma $, $y\in \sigma '$ and
$\sigma \cap \sigma ' \ne\emptyset$, there is a point
$z\in \sigma \cap \sigma '$ with
$ \max\{d(x,z),d(y,z)\}\le C_{1}\, d(x,y)$ for a constant $C_{1}>0$
depending only on $X$ and $d$. Then
    \begin{align} \label{eq:12}
\nonumber      \frac{|f(x)-f(y)|}{d(x,y)}& \le\frac{|f(x)-f(z)|}{d(x,y)}+\frac{|f(z)-f(y)|}{d(x,y)}
      \\
\nonumber      & \le C_{1}\, \frac{|f(x)-f(z)|}{d(x,z)}+C_{1} \,
      \frac{|f(z)-f(y)|}{d(z,y)} \\
     & \le
      2 \, C_{1}\, C_{0} \, (\Lip_{\sigma}(f)+\Lip_{\sigma'}(f)).
    \end{align}
It follows from \eqref{eq:8}, \eqref{eq:9} and \eqref{eq:12} that
  \begin{displaymath}
    \Lip_{K,d}(f)\le C \, \Big(\sup_{\substack{\sigma \in \Pi_{0} \\ \sigma \cap
      K\ne \emptyset}}\Lip_{\sigma }(f)+\|f\|_{C^{0}(K)}\Big)
  \end{displaymath}
  for a constant $C>0$ depending only on $X$ and $d$. The result then
  follows from Lemma \ref{lemm:5}.
\end{proof}

As a consequence we obtain the following stronger version of Theorem
\ref{thm:2}, extending \cite[Theorem 10.4]{ROCK}.

\begin{cor}\label{cor:lipschitz}
  Let $f $ be a weakly concave function on $U$. Then $f$ is Lipschitz
  continuous on any compact subset of $U$.
\end{cor}

The following is the analogue of \cite[Theorem 10.6]{ROCK} for weakly
concave functions on a balanced open subset of a polyhedral space.

\begin{thm}\label{thm:conv1}
  Let $\{f_{i}\}_{i\in I}$ be a family of weakly concave functions on
  $U$ and  $C',C''\subset U$  dense subsets such that
  $\inf_{i\in I} f_{i}(x)$ and $\sup_{i\in I} f_{i}(x)$ are finite for
  every $x\in C'$ and $x\in C''$, respectively.  Then
  $\{f_{i}\}_{i\in I}$ is uniformly bounded and equi-Lipschitz on any
  compact subset of $U$.
\end{thm}

\begin{proof} Let $K\subset U$ be a compact subset and $V\subset U$ an
  open subset with compact closure containing $K$. By Theorem
  \ref{thm:3}, if the family $\{f_{i}\}_{i\in I}$ is uniformly bounded
  on $V$ then it is equi-Lipschitz in $K$. So we only need to prove
  that this family is uniformly bounded on $V$.

  For $x\in U$ set $f(x)=\inf_{i}f_{i}(x) \in \R\cup \{-\infty\}$.
  Let $\Pi $ be a polyhedral complex on $X$. For every
  $\tau \in \Pi(n) $, the functions $f_{i}|_{\tau }$ are concave in
  the classical sense and since $C'$ is dense in $\tau$, \cite[Theorem
  10.6]{ROCK} implies that $f(x)>-\infty $ for $x\in \relint(\tau)$.

  For $x\in X$, there is a balanced convex combination
  $ (x,(x_{i}),(\nu _{i}))$ such that the points $x_{i}$ belong
  to the relative interior of a polyhedron in $\Pi(n)$. Therefore
  \begin{displaymath}
    f(x)\ge \sum_{i}\nu _{i}\, f(x_{i}) > -\infty.
  \end{displaymath}
  By Proposition \ref{prop:23}, $f$ is a weakly concave function on
  $U$ and by Theorem \ref{thm:2} it is continuous on $U$. Hence $f$ is
  uniformly bounded in $K$, which implies that the family
  $\{f{i}\}_{i\in I}$ is uniformly bounded below in $K$.

  Making $K$ bigger and $\Pi $ finer, we can assume that $K$ is a
  finite union of polyhedra of $\Pi $ of maximal dimension contained
  in $V$. For each $\tau\in \Pi(n)$ with $\tau \subset K$ there is a
  point $x\in C''\cap \relint(\tau )$.  We then follow the proof of
  \cite[Theorem 10.6]{ROCK}. Choose $\varepsilon $ such that
  $x + \varepsilon \, \B\subset \tau $, where $\B$ denotes the unit
  ball. Since $x\in C''$, there is $\alpha \in \R$ such that
  $f_{i}(x)\le \alpha $ for all $i$. Moreover, since we have already
  proved the existence of a uniform lower bound, there is also
  $\beta \in \R$ such that $f(z)\ge \beta $ for all $z\in \tau $.

  Let now $y\in \tau$ with $y\not = x$. Set
  \begin{displaymath}
    z= x + \frac{\varepsilon}{\|x-y\|}(x-y) \and
    \lambda = \frac{\varepsilon}{\varepsilon +\|x-y\|}.
  \end{displaymath}
  Then $z\in x + \varepsilon \, \B\subset \tau$ and $x=(1-\lambda
  )z+\lambda y$. Hence, by concavity,
  \begin{displaymath}
    \alpha \ge f_{i}(x)\ge (1-\lambda) f_{i}(z)+\lambda f_{i}(y) \ge
    (1-\lambda )\beta + \lambda f_{i}(y).
  \end{displaymath}
  Thus
  \begin{displaymath}
    f_{i}(y)\le \frac{\alpha -(1-\lambda )\beta}{\lambda}=\alpha +\frac{\|x-y\|(\alpha -\beta )}{\beta} .  
  \end{displaymath}
  Note that the term on the right hand side is a continuous function
  with respect to $y$ on the compact set $\tau $. Therefore there is a
  constant $\gamma _{\tau }$ such that for all $i\in I$ and
  $y\in \tau $ we have that $f_{i}(y)\le \gamma _{\tau }$. Since $K$
  is compact, the number of maximal polyhedra contained in $K$ is
  finite. Hence, we obtain a uniform upper bound on the whole of $K$
  which concludes the proof.
\end{proof}

Theorem \ref{thm:conv1} has the following important consequence,
which is an analogue of \cite[Theorem 10.8]{ROCK} on polyhedral spaces.

\begin{thm}\label{thm:4}
  Let $(f_{i})_{i\ge 0}$ be a sequence of weakly concave functions on
  $U$ and $C\subset U$ a dense subset such that, for all $x\in C$,
  the sequence $f_{i}(x)$ converges to a number $f(x) \in \R$. Then
  the sequence $(f_{i})_{i\ge 0}$ converges to a function $f$ on the
  whole $U$, the limit function $f$ is weakly concave, and the
  convergence is uniform on compact subsets.
\end{thm}

\begin{proof}
  By Theorem \ref{thm:conv1}, the family $(f_{i})_{i\ge 0}$ is
  equi-Lipschitz on each compact subset of $U$. Therefore the
  statement can be proven as in \cite[Theorem
  10.8]{ROCK}.
\end{proof}

On the same spirit, other results of \cite[Chapter 10]{ROCK} can be
translated to the current setting. For instance we particularize the
analogue of \cite[Theorem~10.7]{ROCK}. It is a consequence of theorems
\ref{thm:3} and \ref{thm:conv1}, its proof is similar to the one in
\emph{loc. cit.} and is left to the reader.

\begin{thm}\label{thm:conv2}
  Let $T$ be a locally compact topological space. If
  $f\colon U\times T\to \R$ is a function such that $f(\cdot ,t)$ is
  weakly concave for each $t\in T$ and $f(x,\cdot )$ is continuous for
  each $x\in U$, then $f$ is continuous on $U\times T$.
\end{thm}



\providecommand{\bysame}{\leavevmode\hbox to3em{\hrulefill}\thinspace}
\providecommand{\MR}{\relax\ifhmode\unskip\space\fi MR }
\providecommand{\MRhref}[2]{%
  \href{http://www.ams.org/mathscinet-getitem?mr=#1}{#2}
}
\providecommand{\href}[2]{#2}

\end{document}

%% file: triangle.pdf_t
\begin{picture}(0,0)%
\includegraphics{triangle.pdf}%
\end{picture}%
\setlength{\unitlength}{3947sp}%
\begingroup\makeatletter\ifx\SetFigFont\undefined%
\gdef\SetFigFont#1#2#3#4#5{%
  \reset@font\fontsize{#1}{#2pt}%
  \fontfamily{#3}\fontseries{#4}\fontshape{#5}%
  \selectfont}%
\fi\endgroup%
\begin{picture}(2874,2458)(6492,-4367)
\put(8169,-2618){\makebox(0,0)[b]{\smash{{\SetFigFont{12}{14.4}{\rmdefault}{\mddefault}{\updefault}{\color[rgb]{0,0,0}$N$}%
}}}}
\put(8721,-4230){\makebox(0,0)[b]{\smash{{\SetFigFont{12}{14.4}{\rmdefault}{\mddefault}{\updefault}{\color[rgb]{0,0,0}$M$}%
}}}}
\put(7580,-3798){\makebox(0,0)[b]{\smash{{\SetFigFont{12}{14.4}{\rmdefault}{\mddefault}{\updefault}{\color[rgb]{0,0,0}$A$}%
}}}}
\put(8238,-3797){\makebox(0,0)[b]{\smash{{\SetFigFont{12}{14.4}{\rmdefault}{\mddefault}{\updefault}{\color[rgb]{0,0,0}$B$}%
}}}}
\put(7921,-3211){\makebox(0,0)[b]{\smash{{\SetFigFont{12}{14.4}{\rmdefault}{\mddefault}{\updefault}{\color[rgb]{0,0,0}$C$}%
}}}}
\put(7138,-4213){\makebox(0,0)[b]{\smash{{\SetFigFont{12}{14.4}{\rmdefault}{\mddefault}{\updefault}{\color[rgb]{0,0,0}$L$}%
}}}}
\end{picture}%